\documentclass{amsart}

\usepackage{amsmath,amssymb,amsthm,enumerate}

\usepackage{graphicx,tikz}
\newtheorem{theorem}{Theorem}[section]

\newtheorem*{corollary}{Corollary}

\newtheorem{lemma}[theorem]{Lemma}

\DeclareMathOperator*{\argmin}{argmin}
\DeclareMathOperator*{\argmax}{argmax}

\theoremstyle{definition}

\theoremstyle{remark}
\newtheorem{remark}[theorem]{Remark}

\begin{document}

\title[]{A common variable minimax theorem for graphs}

\thanks{R.R.C is supported by the NIH (5R01NS100049-04). N.F.M. is supported by
the NSF (DMS-1903015). S.S. is supported by the NSF (DMS-2123224) and the
Alfred P. Sloan Foundation.}

\author[]{Ronald R. Coifman}
\address{Department of Mathematics, Yale University, New Haven}
\email{coifman@math.yale.edu}

\author[]{Nicholas F. Marshall}
\address{Department of Mathematics, Princeton University, Princeton}
\email{nicholas.marshall@princeton.edu}

\author[]{Stefan Steinerberger}
\address{Department of Mathematics, University of Washington, Seattle}
\email{steinerb@uw.edu}

\begin{abstract} Let $\mathcal{G} = \{G_1 = (V, E_1), \dots, G_m = (V, E_m)\}$ be a collection 
of $m$ graphs defined on a common set of vertices $V$ but with different edge sets $E_1, \dots, E_m$.
Informally, a function 
$f :V \rightarrow \mathbb{R}$ is smooth with respect to $G_k = (V,E_k)$  
if $f(u) \sim f(v)$ whenever $(u, v) \in E_k$. 
We study the problem of understanding whether there exists a nonconstant function that is smooth with 
respect to all graphs in $\mathcal{G}$, simultaneously, and how to find it if it exists.
\end{abstract}
\maketitle

\section{Introduction}
\subsection{Introduction} \label{intro} Let $G = (V,E)$ be a graph; loosely
speaking, a function $f :V \rightarrow \mathbb{R}$ is smooth with respect to
$G$ if it varies little over adjacent vertices meaning that $f(u) \sim f(v)$
whenever $(u,v) \in E$. Let $\mathcal{G}$ be a collection of $m$ graphs on the
same set of vertices $V$
$$
\mathcal{G} = \{ G_1=(V, E_1),  \ldots, G_m = (V, E_m) \}.
$$
We consider the following problem: among all mean zero unit norm
functions $f : V \rightarrow \mathbb{R}$ which is the smoothest with respect
to $\mathcal{G}$ (see \S \ref{collection} for a formal statement)?

\begin{figure}[h!]
\centering
\begin{tabular}{cc}
\includegraphics[width=.47\textwidth]{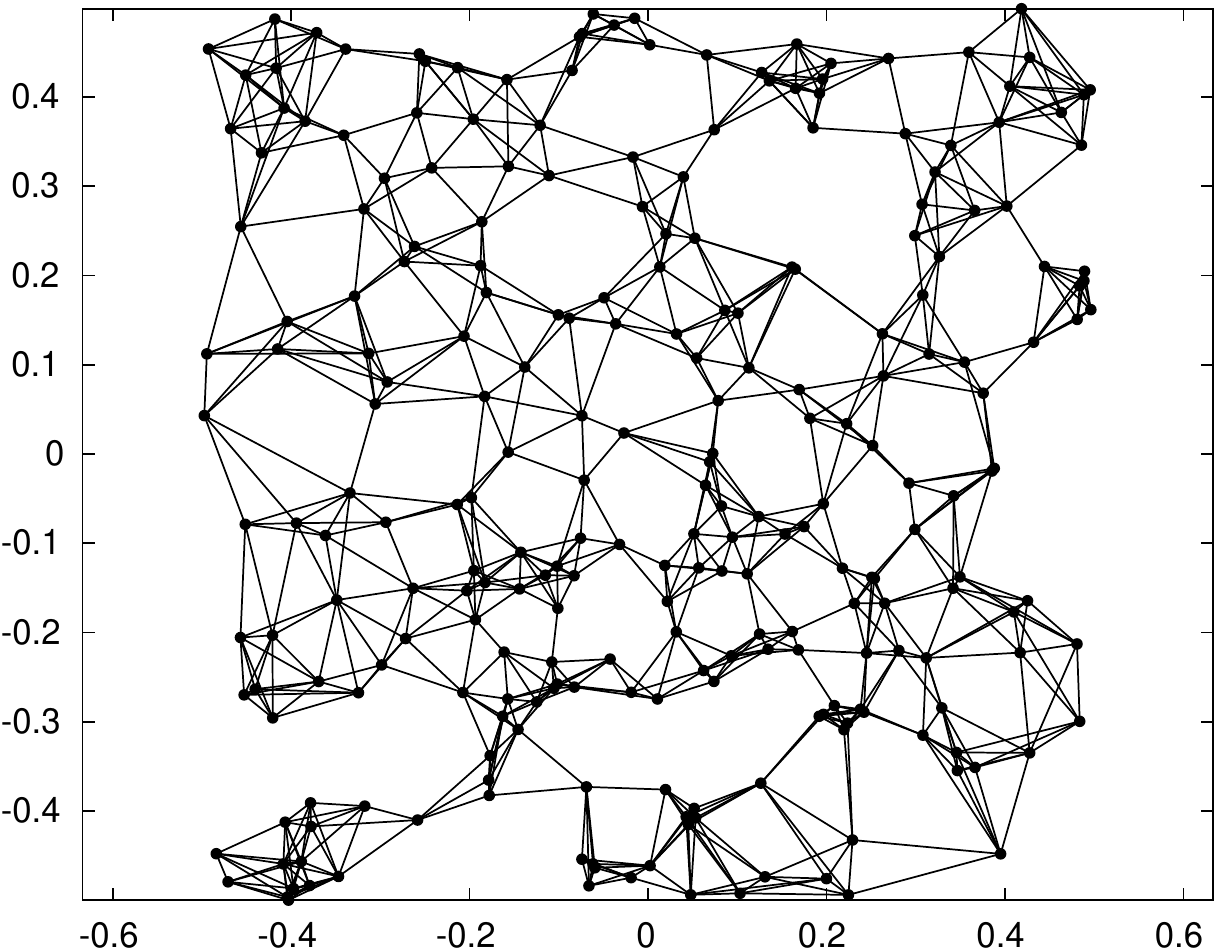} &
\includegraphics[width=.47\textwidth]{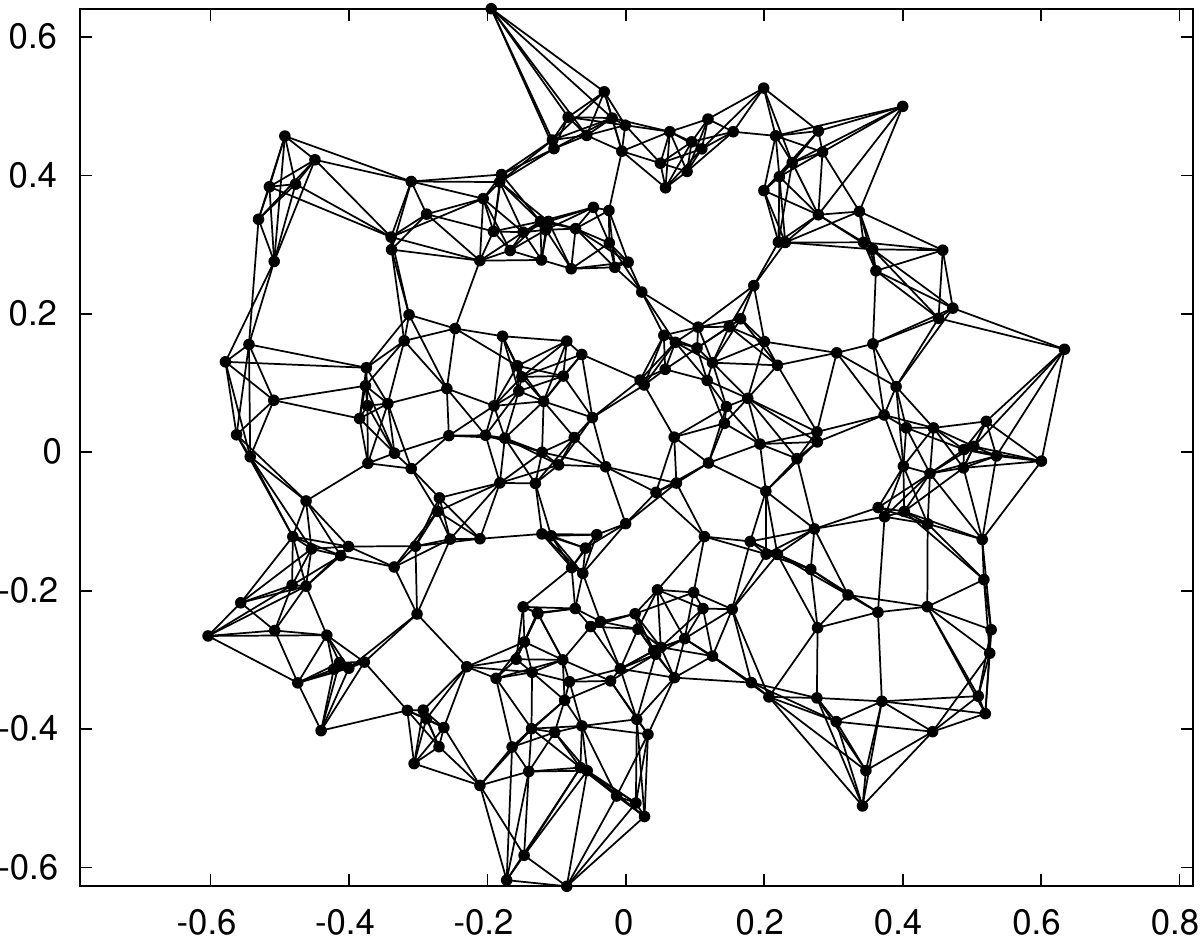} 
\end{tabular}
\caption{A 6-nearest neighbor graph of points in the plane
(left), and a 6-nearest neighbor graph for the same points after each point has
been independently randomly rotated about the origin (right). As the number
of points $n \rightarrow \infty$, commonly smooth functions $f : V \rightarrow
\mathbb{R}$ are functions of the distance to the origin.
}  \label{fig01}
\end{figure}

\subsection{Motivating example}
A geometric example is shown in Figure \ref{fig01}: we are given a set of $n$
uniformly random points in the unit square centered at the origin, and  form a
graph  $G_1 = (V,E_1)$ by connecting each point to its 6-nearest neighbors with
respect to Euclidean distance.  A second  graph $G_2 = (V,E_2)$ is built on the
same set of points as follows: each point is randomly rotated about
the origin (by independent uniformly random rotations), and the rotated points
are  connected to  their 6-nearest neighbors (see \S \ref{rot2} for a precise
description of this example). Two vertices $u$ and $v$ are close in the graph
$G_1$ if the underlying points are physically close in the plane.  Likewise,
$u$ and $v$ are close in the graph $G_2$ if the rotated version of
the underlying points are close.  It becomes clear that any commonly smooth
function $f:V \rightarrow \mathbb{R}$ must be close to a function that only
depends on the distance of the underlying points to the origin 
(in the usual sense as the number of points $n$ becomes large).
How can we detect these `commonly smooth functions' or `common variables' 
if we do not have access to how the
graphs were constructed?  How can we find them from the graph data alone?

\subsection{Problem statement}
Suppose that $\mathcal{G} = (V, E_k)_{k=1}^m$ is a collection of $m$ graphs on a common set of $V$ vertices.
We address two main problems.
\begin{itemize}
\item Is it possible to detect whether there is a nonconstant commonly `smooth' function on
the vertices $V$ (that is smooth with respect to all $m$ graphs)?
\item Can we determine the `smoothest' nonconstant function on $V$ with respect to the collection of graphs $\mathcal{G}$?
\end{itemize}

The precise nature of these questions will strongly depend on the notion of `smoothness' of a 
function $f :V \rightarrow \mathbb{R}$. The main purpose of our paper is
to define a notion of smoothness inspired by Spectral Graph Theory and to provide an
approach which provably solves both problems in the regime where there truly is a common smooth
variable shared by all graphs in a certain precise sense. What we observe in practice is that the 
method is more broadly applicable. We emphasize that the underlying question could be formalized in many different ways (possibly leading to very different mathematics) and many of them might be interesting.

\subsection{Related results.} The problem of determining a commonly smooth
function for a collection of graphs  appears in  different contexts, perhaps
most frequently in data synthesis. Consider a data synthesis problem where
a fixed set of data points is measured in different ways 
(a multi-view problem). Each measurement of the  data points is encoded in a
graph $G_k = (V,E_k)$ whose vertices $V$ are the fixed data points and whose 
edges $E_k$ are determined by the specific measurement.
The end goal is to synthesize this data to extract intrinsic information. In
particular, is there a common variable (function on $V$) that is related 
to how connections between the data points are formed across all of the graphs?
This is a well-studied problem, we refer to 
\cite{belkin, cai, coif1, coif2, dong, liu, kumar, lederman, tang, yeredor} 
and references therein. We especially emphasize three papers. Ma and Lee
\cite{ma} propose working with a sum of Laplacians -- this is similar to our
approach except for the scaling which is crucial (see below for a longer
discussion).  Eynard, Kovnatsky, Bronstein, Glashoff, and Bronstein \cite{eynard}
also work within a spectral framework, and discuss the problem of simultaneous
diagonalization of Laplacians which is philosophically related to our approach.
Yair, Dietrich, Mulayoff, Talmon, and Kevrekidis \cite{yair} use the same
perspective on smoothness as we will (indeed, their paper directly inspired
ours) -- they compute smooth functions on each graph and then look for vectors
having large correlation with the subspaces of smooth functions.

\subsection{Preliminaries} \label{prelim}
Suppose that $G$ is an undirected connected weighted graph on $n$ vertices
described by an $n \times n$ symmetric nonnegative adjacency matrix
$\mathbf{A}$. We use the notion of a graph Laplacian $\mathbf{L}:\mathbb{R}^{n}
\rightarrow \mathbb{R}^{n}$ throughout the paper; we assume that $\mathbf{L}$
is symmetric, positive semi-definite, and has eigenvalue $0$ (of
 multiplicity $1$)
corresponding to the eigenvector $\mathbf{1}$.
An example of such an
operator is the graph Laplacian
\begin{equation} \label{graphLaplacian}
\mathbf{L} = \mathbf{D} - \mathbf{A},
\end{equation}
where $\mathbf{D}$ is the diagonal matrix whose $i$-th diagonal element $d_{ii}$ is the degree
of the vertex $i \in V$.  We use the notation
$$
0 = \lambda_0(\mathbf{L}) < \lambda_1(\mathbf{L}) \le \cdots \le
\lambda_{n-1}(\mathbf{L}),
$$
to denote the eigenvalues of $\mathbf{L}$, and 
$$
\mathbf{1} = \boldsymbol{\psi}_0, \boldsymbol{\psi}_1,\ldots, \boldsymbol{\psi}_{n-1},
$$
to denote the corresponding eigenvectors which we assume are normalized (so that their $\ell^2$-norm is $1$). When $\mathbf{L}$ is given by \eqref{graphLaplacian}, 
its associated quadratic form can be expressed by
\begin{equation} \label{quadform}
\mathbf{x}^\top \mathbf{L} \mathbf{x} = \sum_{(u,v) \in E} a_{u v} (x_u - x_v)^2,
\end{equation}
where $a_{uv}$ is the weight associated with the edge $(u,v)$. In spectral graph theory, this 
quadratic form is a standard way to measure  the smoothness of a function on a graph. 
In order to use this quadratic form as a smoothness score, 
we restrict our attention to the set $X$ of vectors with mean zero and unit length
$$
X = \{ \mathbf{x} \in \mathbb{R}^n : \mathbf{1}^\top \mathbf{x} = 0 \text{ and }
\mathbf{x}^\top \mathbf{x} =1\}.
$$   
Restricting our attention to $X$ is important since it avoids
trivially smooth functions on the vertices of a graph such as functions 
with a large constant component or functions with a very small norm.
We define a smoothness score $s_{\mathbf{L}} : X \rightarrow [1,\infty)$
by
\begin{equation} \label{smoothness}
s_{\mathbf{L}}(\mathbf{x}) = \frac{1}{\lambda_1(\mathbf{L})}
\mathbf{x}^\top \mathbf{L} \mathbf{x}.
\end{equation}
This normalization ensures that
$
s_{\mathbf{L}}(\mathbf{x}) \ge 1
$
with equality if and only if $\mathbf{x}$ is an eigenvector 
of $\mathbf{L}$  of eigenvalue $\lambda_1(\mathbf{L})$. The reason that normalizing 
$s_\mathbf{L}$ is important, is that
we are going to compare smoothness scores of a given function across different graphs. 
Dividing by $\lambda_1(\mathbf{L})$ is just one reasonable method of normalization; for some 
applications it may be advantageous to normalize $s_{\mathbf{L}}$ differently,  
see Remark \ref{othernorm}. The presented results hold for these alternate normalization strategies, 
as well as more general definitions of $s_{\mathbf{L}}$ whose discussion is delayed until later
in the paper to simplify the exposition, see Remark \ref{gendef}

\subsection{Main results} \label{collection}

Suppose that $\mathcal{G} = \{G_1,\ldots,G_m\}$ is a collection of undirected
connected weighted graphs on a common set of $n$ vertices $V$.
Informally speaking, a common variable for
$\mathcal{G}$ is a function defined on the vertices $V$ which is smooth with respect to the geometry of each graph. More precisely,
we can define a score $s_{\mathcal{G}}$ indicating how smooth (in the minimax sense) the smoothest function with respect to $\mathcal{G}$ is by
\begin{equation} \label{opt1}
s_\mathcal{G} = \min_{\mathbf{x} \in X} \max_{k \in
\{1,\ldots,m\}} s_{\mathbf{L}_k}(\mathbf{x}),
\end{equation}
where $s_\mathbf{L}$ is defined by \eqref{smoothness}.
The score $s_\mathcal{G}$ can be used to understand how much
`common information' is shared by the collection of graphs $\mathcal{G}$. If $\boldsymbol{\psi} \in X$ is an argument that minimizes \eqref{opt1}, then 
we call $\boldsymbol{\psi}$ the smoothest function 
with respect to $\mathcal{G}$ or a common variable of $\mathcal{G}$.

We can now present our main results. Theorem \ref{thm1} provides upper and
lower bounds on $s_\mathcal{G}$ in terms of (explicitly computable) spectral quantities:
we show that the smallest eigenvalue of suitably averaged Laplacians serves
as a lower bound. Theorem \ref{thm1} is complemented by Theorem \ref{thm2} which
shows that the lower and upper bounds are equal under an additional assumption,
and for a suitable choice of parameters. Numerical examples will show that they indeed coincide in practice.

\begin{theorem}[Upper and Lower Bounds] \label{thm1}
Let $\mathcal{G}$ be a collection of graphs satisfying the assumptions in 
\S \ref{prelim}.  For any $\mathbf{t} = (t_1,\ldots,t_m) \in T$, where
\begin{equation} \label{Tdefn}
T := \{ \mathbf{t} \in [0,1]^m : t_1 + \cdots + t_m = 1\},
\end{equation}
define the Laplacian
$\mathbf{L}_{\mathbf{t}}$ by the linear combination
\begin{equation} \label{Lteq}
\mathbf{L}_{\mathbf{t}} =   \sum_{k=1}^m
t_k\frac{ \mathbf{L}_k}{\lambda_1(\mathbf{L}_k)}.
\end{equation}
Then,
$$
\lambda_1(\mathbf{L}_{\mathbf{t}}) \le
 s_\mathcal{G} 
\le  
\max_{k \in \{1,\ldots,m\}} s_{\mathbf{L}_k}(
\boldsymbol{\psi}_1(\mathbf{L}_{\mathbf{t}}))
$$
where $\lambda_1(\mathbf{L}_{\mathbf{t}})$ denotes the second smallest eigenvalue 
of $\mathbf{L}_{\mathbf{t}}$, and 
$\boldsymbol{\psi}_1(\mathbf{L}_{\mathbf{t}})$ denotes a unit length eigenvector associated with $\lambda_1(\mathbf{L}_{\mathbf{t}})$.
\end{theorem}

Theorem \ref{thm1} provides us with spectral upper and lower bounds on $ s_\mathcal{G} $. Theorem \ref{thm2} shows, assuming the first nontrivial eigenvalue is simple, that there is an explicit duality relation which
allows us to find the common variable by solving an eigenvalue optimization problem.

\begin{theorem}[Common Information Minimax Theorem] \label{thm2}
In addition to the hypothesis of Theorem \ref{thm1}, assume that 
$\lambda_1(\mathbf{L}_{\mathbf{t}})$ is always simple:
$$
 \lambda_2( \mathbf{L}_{\mathbf{t}}) > \lambda_1( \mathbf{L}_{\mathbf{t}}), \quad \text{for all} \quad
 \mathbf{t} \in T,
$$
where $T$ is defined by \eqref{Tdefn}. If
\begin{equation} \label{tstar}
\mathbf{t}^* = \argmax_{\mathbf{t} \in T}
\lambda_1(\mathbf{L}_{\mathbf{t}}),
\end{equation}
then
$$
\boldsymbol{\psi}_1(\mathbf{L}_{\mathbf{t^*}}) = 
 \argmin_{\mathbf{x} \in X} \max_{k \in \{1,\ldots,m\}} s_{\mathbf{L}_k}(\mathbf{x})  .
$$
\end{theorem}

The proofs of Theorems \ref{thm1} and \ref{thm2} are given in \S
\ref{proofmain}. In practice, these theorems can be used in
conjunction, see Remark \ref{useboth}. The
optimization problem \eqref{tstar} is straightforward to solve numerically using
standard methods, see \S \ref{numerical}.

\subsection{Diffusion geometry interpretation}
In the following, we describe how Theorems \ref{thm1} and \ref{thm2} can be interpreted in terms of 
diffusion geometry methods.  Given a symmetric positive semi-definite matrix
$\mathbf{L}$  which has eigenvalue $0$ (of multiplicity $1$) associated with the eigenvector
$\mathbf{1}$, we can
define a diffusion (or averaging) operator  $\mathbf{H}^\tau$ by
$$
\mathbf{H}^\tau = \exp(-\tau \mathbf{L}),
$$
where $\exp(\mathbf{A}) = \mathbf{I} + \mathbf{A} + \frac{1}{2!} \mathbf{A} +
\cdots $ is the matrix exponential and $\tau > 0$ plays the role of time. If $\mathbf{L}$ has eigenvalues
$$
0 = \lambda_0(\mathbf{L}) < \lambda_1(\mathbf{L}) \le \cdots \le
\lambda_{n-1}(\mathbf{L}),
$$
then by the spectral mapping theorem $\mathbf{H}^\tau$ has eigenvalues
\begin{equation} \label{spectral}
1 = e^{-\tau \lambda_0(\mathbf{L)}} > e^{-\tau \lambda_1(\mathbf{L})} \ge \cdots
\ge e^{-\tau \lambda_{n-1}(\mathbf{L})}.
\end{equation}
The following corollary is immediate from Theorem \ref{thm2} and \eqref{spectral}.

\begin{corollary}[Diffusion interpretation] \label{cor1}
Under the hypothesis of Theorem \ref{thm2}, define the diffusion operator $\mathbf{H}_{\mathbf{t}}^\tau$ by
$$
\mathbf{H}_{\mathbf{t}}^\tau = \exp \left(  -\tau \sum_{k=1}^m t_k
\frac{\mathbf{L}_k }{\lambda_1(\mathbf{L}_k)}  \right).
$$
Let
$$
\mathbf{t}^* = (t_1^*,\ldots,t_m^*) = \argmin_{\mathbf{t} \in T}
\mu_1(\mathbf{H}_{\mathbf{t}}^\tau),
$$
where $\mu_1(\mathbf{H}_{\mathbf{t}}^\tau)$ is the second largest eigenvalue of
$\mathbf{H}_{\mathbf{t}}^\tau$, and $T$ is defined in \eqref{Tdefn}. Then,
$$
\boldsymbol{\varphi}_1(\mathbf{H}_{\mathbf{t}^*}^\tau) 
= \argmin_{\mathbf{x} \in X} \max_{k \in \{1,\ldots,m\}} s_{\mathbf{L}_k}(\mathbf{x}),
$$
where $\boldsymbol{\varphi}_1(\mathbf{H}_{\mathbf{t}^*}^\tau)$ denotes a unit length
eigenvector associated with  $\mu_1(\mathbf{H}_{\mathbf{t}^*}^\tau)$. 
\end{corollary}

This corollary, which rewrites Theorem \ref{thm2} in terms of a diffusion operator,
has several interesting consequences. 

\begin{enumerate}[\indent 1)]
\item The parameters $\mathbf{t}^* = (t_1^*,\ldots,t_m^*)$ can be interpreted as optimally 
tuned diffusion times for the graphs $G_1,\ldots,G_m$. The operator
$\mathbf{H}_{\mathbf{t}^*}^\tau$ uncovers common information from the graphs by 
optimally diffusing on these graphs at different rates.
\item The operator $\mathbf{H}_{\mathbf{t}^*}^\tau$ can be used to define a
diffusion distance on the common set of vertices $V$ on which the graphs
$\mathcal{G}$ are defined. Assume that $V = \{1,\ldots,n\}$. We can define the diffusion 
distance $D_{\mathbf{t}*}^\tau : V \times V \rightarrow \mathbb{R}$ by
$$
D_{\mathbf{t}^*}^\tau(i,j) = \| \mathbf{H}_{\mathbf{t}^*}^\tau \boldsymbol{\delta}_i -
\mathbf{H}_{\mathbf{t}^*}^\tau \boldsymbol{\delta}_j \|_{\ell^2},
$$
where $\boldsymbol{\delta}_i$ is the column vector whose
$i$-th entry is $1$ and other entries are $0$.
\item For any chosen dimension $d \ge 1$,
the operator
$\mathbf{H}_{\mathbf{t}^*}^\tau$ can be used to define a diffusion map
$\Psi^\tau : V \rightarrow \mathbb{R}^d$ by 
$$
\Psi^\tau(j) = \left( \begin{array}{c} \mu_1(\mathbf{H}_{\mathbf{t}^*}^\tau)
\varphi_{1,j}(\mathbf{H}_{\mathbf{t}^*}^\tau) \\ \vdots \\ \mu_d(\mathbf{H}_{\mathbf{t}^*}^\tau)
\varphi_{d,j}(\mathbf{H}_{\mathbf{t}^*}^\tau)
\end{array} \right),
$$
where $\varphi_{i,j}(\mathbf{H}_{\mathbf{t}^*}^\tau)$ denotes
the $j$-th entry of the eigenvector
$\boldsymbol{\varphi}_i(\mathbf{H}_{\mathbf{t}^*}^\tau)$ associated with the
eigenvalue $\mu_i(\mathbf{H}_{\mathbf{t}^*}^\tau) =
e^{-\tau \lambda_i(\mathbf{L}_{\mathbf{t}^*})}$. 
\end{enumerate}

\begin{remark}[Using Theorems \ref{thm1} and \ref{thm2} in conjunction] \label{useboth}
In applications, the optimization problem \eqref{tstar} for $\mathbf{t}^*$ in
Theorem \ref{thm2} can be solved numerically using gradient based
optimization methods and Lemma \ref{lemg}. 
Suppose that $\tilde{\mathbf{t}}$ is the numerical solution to
\eqref{tstar}, and let  $\tilde{\boldsymbol{\psi}}_1$ denote a normalized
eigenvector corresponding to $\lambda_1(\mathbf{L}_{\tilde{\mathbf{t}}})$.
By Theorem \ref{thm1} we have the following error estimate:
\begin{equation} \label{errestt}
\left| \max_{k \in \{1,\ldots,m\}} s_{\mathbf{L}_k}(
\tilde{\boldsymbol{\psi}}_1)
-  s_\mathcal{G}\right|  \le 
 \max_{k \in \{1,\ldots,m\}} s_{\mathbf{L}_k}(
\tilde{\boldsymbol{\psi}}_1)
- \lambda_1( \mathbf{L}_{\tilde{\mathbf{t}}}).
\end{equation}
This inequality can be used to verify that the numerical optimization is successful. 
For each of our numerical examples
presented in \S \ref{numerical} we use \eqref{errestt} to verify that 
we are able to accurately solve each optimization problem (with error $\lesssim 10^{-6})$;
however, in practice \eqref{errestt} could also be used to stop the optimization process 
when the error is less than, say, $10^{-1}$, if that is sufficient for the application.
Furthermore, we note that since Theorem \ref{thm1} does not
require the eigenvalue multiplicity condition of
Theorem \ref{thm2} and Lemma \ref{lemg}, this error estimate can be used to
check that a candidate  common variable $\boldsymbol{\psi}_1$ determined using Theorem \ref{thm2} is close to optimal 
without having to verify that the spectral gap hypothesis of Theorem 
\ref{thm2} holds.
\end{remark}

\begin{remark}[Alternate normalization methods] \label{othernorm}
Recall that we defined the smoothness score $s_{\mathbf{L}} : X \rightarrow [1,\infty)$
by
$$
s_{\mathbf{L}}(\mathbf{x}) = \frac{1}{\lambda_1(\mathbf{L})}
\mathbf{x}^\top \mathbf{L} \mathbf{x},
$$
such that the `smoothest function' with respect to $s_{\mathbf{L}}$
has smoothness score 1.
In practice it may be advantageous to normalize
the quadratic form differently. For example, we could define
$a_{\mathbf{L}} : X \rightarrow (0,\infty)$  by
\begin{equation} \label{alpha2}
a_{\mathbf{L}}(\mathbf{x}) =  \left( \frac{1}{n-1} \sum_{j=1}^{n-1} \lambda_j(\mathbf{L}) \right)^{-1} \mathbf{x}^\top \mathbf{L} \mathbf{x},
\end{equation}
such that the `average smoothness' with respect to $a_{\mathbf{L}}$ is $1$. More precisely, with
this definition
we have
 $\mathbb{E}[ a_{\mathbf{L}}(\mathbf{x})] = 1$, where the
expectation is taken over $\mathbf{x}$ chosen uniformly at random from $X$.
Indeed,
by writing $\mathbf{x} = \sum_{j=1}^{n-1} c_j \boldsymbol{\psi}_j$, where
$(c_1,\ldots,c_{n-1})$ is chosen uniformly at random with respect to the surface measure on $ \mathbb{S}^{n-2}$ we have
$$
\mathbb{E}[ \mathbf{x}^\top \mathbf{L} \mathbf{x}]
=  \mathbb{E}\left[ \sum_{j=1}^{n-1} c_j^2
\lambda_j(\mathbf{L})  \right]
=  \sum_{j=1}^{n-1} \frac{\lambda_j(\mathbf{L})}{n-1}.
$$
Other methods of normalization are conceivable: one could, for
example, consider decreasing weights that put more emphasis on lower
frequencies. Our theoretical results are independent of the choice of
normalization. However, for applications the distinction
between choosing to normalize based on the `smoothest function' or `average
smoothness' (or some other intermediate 
normalization method) may be important; we provide such an example in \S \ref{spiraltorus}.

\end{remark}


\section{Proof of main results} \label{proofmain}
We start by proving Theorem \ref{thm1}. After that, we establish Lemma \ref{lemg} and use it to establish Theorem \ref{thm2}.
\subsection{Proof of Theorem \ref{thm1}}
\begin{proof}
Recall that
$$
X = \{ \mathbf{x} \in \mathbb{R}^n : \mathbf{1}^\top \mathbf{x} = 0 \text{ and }
\mathbf{x}^\top \mathbf{x} = 1\}.
$$ 
We have
$$
s_\mathcal{G} = \min_{\mathbf{x} \in X} \max_{k \in \{1,\ldots,m\}}
s_{\mathbf{L}_k}(\mathbf{x}) = 
\min_{\mathbf{x} \in X} \max_{\mathbf{t} \in T}
\sum_{k=1}^m t_k s_{\mathbf{L}_k}(\mathbf{x}),
$$
where 
$$
T = \left\{ \mathbf{t} = (t_1,\ldots,t_m) \in [0,1]^m : 
t_1 + \cdots + t_m = 1 \right\}.
$$
For any fixed $\mathbf{t} \in T$ (not depending on
$\mathbf{x}$) we have
$$
\min_{\mathbf{x} \in X} \max_{\mathbf{t} \in T} \sum_{k=1}^m t_k
s_{\mathbf{L}_k}(\mathbf{x}) \ge  \min_{\mathbf{x} \in X}
\sum_{k=1}^m t_k s_{\mathbf{L}_k}(\mathbf{x}) = 
\min_{\mathbf{x} \in X} \mathbf{x}^\top \mathbf{L}_\mathbf{t} \mathbf{x},
$$
where
$$
\mathbf{L}_\mathbf{t} = \sum_{k=1}^m \frac{t_k}{\lambda_1(\mathbf{L}_k)}
\mathbf{L}_k .
$$
By the Courant-Fischer Theorem
$$
 \min_{\mathbf{x} \in X} \mathbf{x}^\top \mathbf{L}_\mathbf{t}
\mathbf{x}  =  \lambda_1(\mathbf{L}_{\mathbf{t}}).
$$
In combination, the above inequalities give
$$
s_\mathcal{G}  \ge \lambda_1(\mathbf{L}_{\mathbf{t}}).
$$
Let $\boldsymbol{\psi}_1(\mathbf{L}_{\mathbf{t}})$ be a unit length
eigenvector corresponding to $\lambda_1(\mathbf{L}_\mathbf{t})$. 
By using this eigenvector as a test vector we have
$$s_\mathcal{G} = \min_{\mathbf{x} \in X} \max_{k \in
\{1,\ldots,m\}} s_{\mathbf{L}_k}(\mathbf{x}) \leq \max_{k \in\{1,\ldots,m\}} s_{\mathbf{L}_k}(\boldsymbol{\psi}_1(\mathbf{L}_{\mathbf{t}}))
$$
This completes the proof.
\end{proof}
\begin{lemma}[Gradient of eigenvalue] \label{lemg}
Suppose that the assumptions of Theorem \ref{thm2} hold.
Suppose that $\mathbf{t}' = (t_1,\ldots,t_{m-1})$ and assume $t_m := 1 - (t_1 + \ldots + t_{m-1})$. We have
$$
\nabla_{\mathbf{t}'} \lambda_1(\mathbf{L}_{\mathbf{t}}) = \left( 
\begin{array}{c}
\boldsymbol{\psi}_1(
\mathbf{L}_{\mathbf{t}})^\top \left( 
\frac{1}{\lambda_1(\mathbf{L}_1)}
\mathbf{L}_1 - 
 \frac{1}{\lambda_1(\mathbf{L}_m)}
\mathbf{L}_{m} \right)\boldsymbol{\psi}_1(\mathbf{L}_{\mathbf{t}}) \\
\vdots \\
\boldsymbol{\psi}_1(\mathbf{L}_{\mathbf{t}})^\top \left( 
\frac{1}{\lambda_1(\mathbf{L}_{m-1})}
\mathbf{L}_{m-1} - 
\frac{1}{\lambda_1(\mathbf{L}_{m})}
\mathbf{L}_{m} \right) \boldsymbol{\psi}_1(\mathbf{L}_{\mathbf{t}})
\end{array} \right).
$$
Moreover,  $\lambda_1(\mathbf{L}_{\mathbf{t}})$ is equal to $s_\mathcal{G}$ whenever the gradient vanishes.
\end{lemma}
\begin{proof}
Suppose that $\mathbf{t}' = (t_1,\ldots,t_{m-1})$ and assume $t_m := 1 -
(t_1+\ldots+t_{m-1})$. Under this assumption, we use the notation
$\mathbf{L}_{\mathbf{t}'} = \mathbf{L}_{\mathbf{t}}$ interchangeably.
We will prove that
$$
\frac{\partial}{\partial t_j} \lambda_1(\mathbf{L}_{\mathbf{t}'}) =
\boldsymbol{\psi}_1(\mathbf{L}_{\mathbf{t}'})^\top \left(
\frac{1}{\lambda_1(\mathbf{L}_j)}
\mathbf{L}_j - 
\frac{1}{\lambda_1(\mathbf{L}_m)}
\mathbf{L}_{m} \right) \boldsymbol{\psi}_1(\mathbf{L}_{\mathbf{t}'}),
$$
for $j \in \{1,\ldots,m-1\}$.
Let $\mathbf{e}_j \in \mathbb{R}^{m-1}$ be the $j$-th standard basis vector
(whose $j$-th entry is $1$ and 
other entries are $0$). For $\varepsilon > 0$ we have
$$
\lambda_1(\mathbf{L}_{\mathbf{t}' + \varepsilon \mathbf{e}_j }) = 
\boldsymbol{\psi}_1(\mathbf{L}_{\mathbf{t}' + \varepsilon \mathbf{e}_j })^\top 
(\mathbf{L}_{\mathbf{t}' + \varepsilon \mathbf{e}_j})
\boldsymbol{\psi}_1(\mathbf{L}_{\mathbf{t}' + \varepsilon \mathbf{e}_j }).
$$
By the definition of $\mathbf{L}_{\mathbf{t}'}$ we have
\begin{equation} \label{leps}
\mathbf{L}_{\mathbf{t}' + \varepsilon \mathbf{e}_j} = \mathbf{L}_{\mathbf{t}'} + \varepsilon 
\left( \frac{1}{\lambda_1(\mathbf{L}_j)} \mathbf{L}_j
- \frac{1}{\lambda_1(\mathbf{L}_m)} \mathbf{L}_m \right).
\end{equation}
We will now argue that it is possible to express the first normalized
eigenvector of the perturbed matrix $\mathbf{L}_{\mathbf{t}' + \varepsilon
\mathbf{e}_j}$ as a small perturbation of the first normalized eigenvector of
the unperturbed matrix $\mathbf{L}_{\mathbf{t}'}$. Using that all our
eigenvectors are defined to be normalized, we can write
$$
\boldsymbol{\psi}_1(\mathbf{L}_{\mathbf{t}' + \varepsilon \mathbf{e}_j})
= c \boldsymbol{\psi}_1(\mathbf{L}_{\mathbf{t}'}) + \boldsymbol{\delta},
$$
where $c$ is a coefficient and $\boldsymbol{\delta}$ is orthogonal to 
$\boldsymbol{\psi}_1(\mathbf{L}_{\mathbf{t}'})$. It remains to show that $c$ is
close to 1 or, equivalently, that $\boldsymbol{\delta}$ is small. For this, we
use the Davis-Kahan theorem. If $\theta$ denotes the angle between
$\boldsymbol{\psi}_1(\mathbf{L}_{\mathbf{t}'})$ and
$\boldsymbol{\psi}_1(\mathbf{L}_{\mathbf{t}'+ \varepsilon \mathbf{e}_j})$
then by the Davis-Kahan theorem
\begin{equation} \label{davisk}
 \sin(\theta) \le  \frac{2 \| 
\mathbf{L}_{\mathbf{t}' + \varepsilon \mathbf{e}_j} - 
\mathbf{L}_{\mathbf{t}'} \|}{\min_{i \not = 1} |\lambda_1( \mathbf{L}_{\mathbf{t}'})
- \lambda_i(\mathbf{L}_{\mathbf{t}'})|}.
\end{equation}
The denominator is uniformly bounded away from 0 as part of the assumptions of
Theorem \ref{thm2}, which are assumed to hold in the statement of the lemma.
Therefore, combining \eqref{leps} and \eqref{davisk} yields
$$
 \sin(\theta) \le  \mathcal{O}\left( \| 
\mathbf{L}_{\mathbf{t}' + \varepsilon \mathbf{e}_j} - 
\mathbf{L}_{\mathbf{t}'} \| \right) = \mathcal{O}(\varepsilon).
$$
We can now compute the cosine of $\theta$ via an inner product and obtain
$$ \cos{(\theta)} = \left\langle \boldsymbol{\psi}_1(\mathbf{L}_{\mathbf{t}'}), \boldsymbol{\psi}_1(\mathbf{L}_{\mathbf{t}' + \varepsilon \mathbf{e}_j}) \right\rangle = c.$$
Using
$$ 1= \| \boldsymbol{\psi}_1(\mathbf{L}_{\mathbf{t}' + \varepsilon \mathbf{e}_j}) \|^2 = c^2 + \|\boldsymbol{\delta}\|^2,$$
we arrive at
$$ \| \boldsymbol{\delta}\|^2 = 1 - c^2 = 1-\cos{(\theta)}^2 = \sin{(\theta)}^2 = \mathcal{O}(\varepsilon^2)$$
from which we deduce $\| \boldsymbol{\delta} \| = \mathcal{O}(\varepsilon)$. Using these identities, we can perform an expansion of
$\lambda_1(\mathbf{L}_{\mathbf{t}' + \varepsilon \mathbf{e}_j })$ up to order $\varepsilon$. We start by writing
$$
\lambda_1(\mathbf{L}_{\mathbf{t}' + \varepsilon \mathbf{e}_j }) = 
(c \boldsymbol{\psi}_1(\mathbf{L}_{\mathbf{t}'}) + \boldsymbol{\delta})^\top 
\left(
\mathbf{L}_{\mathbf{t}'} + \varepsilon 
\left( \frac{1}{\lambda_1(\mathbf{L}_j)} \mathbf{L}_j
- \frac{1}{\lambda_1(\mathbf{L}_m)} \mathbf{L}_m \right) \right)
(c \boldsymbol{\psi}_1(\mathbf{L}_{\mathbf{t}'}) + \boldsymbol{\delta}).
$$
Recalling that $c^2 = 1- \mathcal{O}(\varepsilon^2)$,  expanding the right hand side gives
\begin{multline*}
\lambda_1(\mathbf{L}_{\mathbf{t}' + \varepsilon \mathbf{e}_j }) = 
c^2 \boldsymbol{\psi}_1(\mathbf{L}_{\mathbf{t}'})^\top  
\mathbf{L}_{\mathbf{t}'}
 \boldsymbol{\psi}_1(\mathbf{L}_{\mathbf{t}'})
+ 2c \boldsymbol{\delta} \mathbf{L}_{\mathbf{t}'} \boldsymbol{\psi}_1(\mathbf{L}_{\mathbf{t}'})
\\ +
\varepsilon 
\boldsymbol{\psi}_1(\mathbf{L}_{\mathbf{t}'})^\top  
\left( \frac{1}{\lambda_1(\mathbf{L}_j)} \mathbf{L}_j
- \frac{1}{\lambda_1(\mathbf{L}_m)} \mathbf{L}_m \right) 
\boldsymbol{\psi}_1(\mathbf{L}_{\mathbf{t}'}) + \mathcal{O}(\varepsilon^2)
\end{multline*}
The first term on the right hand side is equal to $\lambda_1(\mathbf{L}_\mathbf{t}') 
+ \mathcal{O}(\varepsilon^2)$, and the second term is equal to zero since
$\boldsymbol{\delta}$ is orthogonal  to the eigenvector
$\boldsymbol{\psi}_1(\mathbf{L}_\mathbf{t}')$. It follows that
$$
\frac{
\lambda_1(\mathbf{L}_{\mathbf{t}' + \varepsilon \mathbf{e}_j }) -
\lambda_1(\mathbf{L}_\mathbf{t}') }{\varepsilon} 
= 
\boldsymbol{\psi}_1(\mathbf{L}_{\mathbf{t}'})^\top  
\left( \frac{1}{\lambda_1(\mathbf{L}_j)} \mathbf{L}_j
- \frac{1}{\lambda_1(\mathbf{L}_m)} \mathbf{L}_m \right) 
\boldsymbol{\psi}_1(\mathbf{L}_{\mathbf{t}'}) + \mathcal{O}(\varepsilon).
$$
This argument works for all $j \in \{1,\ldots,m-1\}$ so the proof is complete.
\end{proof}

\subsection{Proof of Theorem \ref{thm2}}
\begin{proof}
Suppose that
$$
\mathbf{t}^* = \argmax_{\mathbf{t} \in T}
\lambda_1(\mathbf{L}_{\mathbf{t}}).
$$
We use the notation $\mathbf{t}' = (t_1,\ldots,t_{m-1})$ where $t_m := 1 - (t_1 + \cdots + t_{m-1})$.
First, consider the case where $\mathbf{t}^*$ is contained in the interior of $T$. 
In this case, 
$$
\nabla_{\mathbf{t}'} \lambda_1( \mathbf{L}_{\mathbf{t}} ) |_{ \mathbf{t} = \mathbf{t}^*} = 0.
$$
Thus, by Lemma \ref{lemg} we have
$$
\boldsymbol{\psi}_1(
\mathbf{L}_{\mathbf{t}^*})^\top \left( 
\frac{1}{\lambda_1(\mathbf{L}_k)}
\mathbf{L}_k - 
 \frac{1}{\lambda_1(\mathbf{L}_m)}
\mathbf{L}_{m} \right)\boldsymbol{\psi}_1(\mathbf{L}_{\mathbf{t}^*}) = 0,
$$
for $k \in \{1,\ldots,m-1\}$, since this equation can be equivalently written as
$$  
s_{\mathbf{L}_k}(
\boldsymbol{\psi}_1(\mathbf{L}_{\mathbf{t}^*})) - s_{\mathbf{L}_m}(
\boldsymbol{\psi}_1(\mathbf{L}_{\mathbf{t}^*})) = 0,
$$
it follows that all of these quadratic forms are equal:
$$
s_{\mathbf{L}_1}(\boldsymbol{\psi}_1(\mathbf{L}_{\mathbf{t}^*})) = \cdots =
s_{\mathbf{L}_m}(\boldsymbol{\psi}_1(\mathbf{L}_{\mathbf{t}^*})).
$$
Informally speaking, the smoothest function or common variable
$\boldsymbol{\psi}_1(\mathbf{L}_{\mathbf{t}^*})$ is
 indifferent between the different smoothness measures
$s_{\mathbf{L}_1},\ldots,s_{\mathbf{L}_m}$.
Thus, 
$$
\lambda_1(\mathbf{L}_{\mathbf{t}^*})
= \max_{k \in \{1,\ldots,m\}} s_{\mathbf{L}_k}(
\boldsymbol{\psi}_1(\mathbf{L}_{\mathbf{t}^*})).
$$
We recall Theorem \ref{thm1} states that
$$
\lambda_1(\mathbf{L}_{\mathbf{t}}) \le
 s_\mathcal{G} 
\le  
\max_{k \in \{1,\ldots,m\}} s_{\mathbf{L}_k}(
\boldsymbol{\psi}_1(\mathbf{L}_{\mathbf{t}})),
$$
from which we can conclude that
$$
\max_{k \in \{1,\ldots,m\}} s_{\mathbf{L}_k}(
\boldsymbol{\psi}_1(\mathbf{L}_{\mathbf{t}^*})) = s_\mathcal{G}.
$$
It remains to consider the case where $\mathbf{t}^*$ is not contained in the
interior of $T$.  Without loss of generality, suppose that $t_1^*=\cdots=t_p^*
=0$ and $t_{p+1}^*,\ldots,t_m^* \not = 0$. Suppose first that $p+1 \neq m$.
Then there are at least $2$ positive entries and, in particular, $0<t_m<1$. We
can thus apply Lemma \ref{lemg} and conclude that for 
$k \in \{p+1,\ldots, m\}$
$$
\boldsymbol{\psi}_1(
\mathbf{L}_{\mathbf{t}^*})^\top \left( 
\frac{1}{\lambda_1(\mathbf{L}_k)}
\mathbf{L}_k - 
 \frac{1}{\lambda_1(\mathbf{L}_m)}
\mathbf{L}_{m} \right)\boldsymbol{\psi}_1(\mathbf{L}_{\mathbf{t}^*}) = 0
$$
which, as above, is equivalent to,
\begin{equation} \label{eqp1}
s_{\mathbf{L}_k}(
\boldsymbol{\psi}_1(\mathbf{L}_{\mathbf{t}^*})) =  s_{\mathbf{L}_m}(
\boldsymbol{\psi}_1(\mathbf{L}_{\mathbf{t}^*})),
\quad \text{for} \quad k \in \{p+1,\ldots, m\}.
\end{equation}
If $p+1 = m$, then  \eqref{eqp1} holds trivially.
It remains to deal with the entries $t_1^*, \dots, t_{p}^*$ (which are all 0).
Fix $k \in \{1,\ldots,p\}$. Since $\boldsymbol{t}^*$ is maximal,
the derivative of $\lambda_1(\mathbf{L}_{\mathbf{t}'})$ at $\mathbf{t}' =
(t_1^*,\ldots,t_{m-1}^*)$ in the direction $\mathbf{e}_k$ must be negative 
and thus by Lemma \ref{lemg}
$$  
s_{\mathbf{L}_k}(
\boldsymbol{\psi}_1(\mathbf{L}_{\mathbf{t}^*})) \leq  s_{\mathbf{L}_m}(
\boldsymbol{\psi}_1(\mathbf{L}_{\mathbf{t}^*})).$$
It follows that
$$
\max_{k \in \{1,\ldots,m\}} s_{\mathbf{L}_k}(
\boldsymbol{\psi}_1(\mathbf{L}_{\mathbf{t}^*})) = 
\max_{k \in \{p+1,\ldots,m\}} s_{\mathbf{L}_k}(
\boldsymbol{\psi}_1(\mathbf{L}_{\mathbf{t}^*})) =  \lambda_1(\mathbf{L}_{\mathbf{t}^*}).
$$
Appealing to Theorem \ref{thm1} once more gives
$$
\lambda_1(\mathbf{L}_{\mathbf{t}^*}) \le
 s_\mathcal{G} 
\le  
\max_{k \in \{1,\ldots,m\}} s_{\mathbf{L}_k}(
\boldsymbol{\psi}_1(\mathbf{L}_{\mathbf{t}}^*)),
$$
we can conclude that
$$ \lambda_1(\mathbf{L}_{\mathbf{t}^*}) =
 s_\mathcal{G} 
=  
\max_{k \in \{1,\ldots,m\}} s_{\mathbf{L}_k}(
\boldsymbol{\psi}_1(\mathbf{L}_{\mathbf{t}^*})).
$$
This completes the proof.
\end{proof}

\section{Numerical examples} \label{numerical}

\subsection{The Laplacian} Our approach is completely general with respect to the underlying notion of Laplacian $\mathbf{L}$ and many different types of Laplacians could be used. 
We merely require that $\mathbf{L}$ is symmetric positive semi-definite, and that $\mathbf{L}$ has 
eigenvalue  $0$ of multiplicity $1$ (corresponding to constant functions). 
For the purpose of consistency, all examples will be computed using the 
bi-stochastic Laplacian which is defined as follows. Assume that
$\mathbf{A}$ is a symmetric non-negative weighted adjacency matrix with a positive main diagonal. 
By using the Sinkhorn-Kopp algorithm (see Lemma \ref{lemsink}) it is possible to determine a 
symmetric positive definite diagonal matrix $\mathbf{D}$ such that
$$
\mathbf{D}^{-1/2} \mathbf{A} \mathbf{D}^{-1/2} \mathbf{1} = \mathbf{1},
$$
where $\mathbf{1}$ denotes a column vector of ones. 
Given such a matrix $\mathbf{D}$ we define the bi-stochastic graph Laplacian
$\mathbf{L}$ by
$$
\mathbf{L} = \mathbf{I} - \mathbf{D}^{-1/2} \mathbf{A} \mathbf{D}^{-1/2},
$$
where $\mathbf{I}$ is the identity matrix.  The bi-stochastic graph Laplacian
can be viewed as the graph Laplacian of a graph whose weighted adjacency matrix
is  $\mathbf{D}^{-1/2} \mathbf{A} \mathbf{D}^{-1/2}$, and thus the bi-stochastic
graph Laplacian has the same properties as the graph Laplacian discussed in \S
\ref{prelim}. We refer to \S \ref{bistochastic} for more details on how to compute the
bi-stochastic Laplacian.

\subsection{Nearest neighbor graph definition}
Let $X = \{x_1,\ldots,x_n\}$ be a subset of $\mathbb{R}^d$. We
say that $N_k(x_j)$ is a set of $k$-nearest neighbors of $x_j$ in $X$ if
$N_k(x_j)$ is a subset of $X \setminus \{x_j\}$ consisting of $k$ points which
has the property
$$
\max_{x \in N_k(x_j)} \|x - x_j\| \le \min_{y \in X \setminus (N_k(x_j) \cup
\{x_j\})} \|y - x_j\|.
$$
We say that $G$ is a $k$-nearest neighbor graph for $X$ if its adjacency matrix $\mathbf{A}
= (a_{ij})$ satisfies
$$
a_{ij} = \left\{ \begin{array}{ll}
1 & \text{if } i=j, \\
1 & \text{if } x_i \in N_k(x_j) \text{ or } x_j \in N_k(x_i), \text{ and}\\
0 & \text{otherwise,}
\end{array} \right. 
$$
for $i,j = 1,\ldots,n$, and for some choice of $k$-nearest neighbors $N_k(x_1),\ldots,N_k(x_n)$.
Note that our definition of a $k$-nearest neighbor
graph includes self loops  for each vertex. This
assumption allows us to perform a bi-stochastic normalization of the adjacency
matrix. We note that assuming that a  graph has self loops is
 a common assumption when working with stochastic matrices on graphs since it
ensures these stochastic matrices are aperiodic.

\subsection{Independent rotations in two dimensions} \label{rot2}

Let $X_1= \{x_1,\ldots,x_n\}$ be a set of $n=250$ independent uniformly
random points from the unit square $[-1/2,1/2]^2$,
and $\theta_1,\ldots,\theta_n$ be independent uniformly random points from
$[0,2\pi)$. Set
$$
X_2 = \{ T_{\theta_1}(x_1),\ldots,T_{\theta_n}(x_n) \},
$$
where $T_\theta(x)$ denotes the rotation of $x$ by angle $\theta$ about the origin.
More precisely, if $x = (r \cos \phi, r \sin \phi)$, then
$T_\theta(x) = (r \cos(\phi + \theta), r \sin(\phi + \theta))$. 
To summarize, the set $X_2$ is created by rotating the points in $X_1$ 
about the origin with independent uniformly random rotations. Let $G_1$ and 
$G_2$ be $6$-nearest neighbor graphs of $X_1$ and $X_2$, respectively, 
see  Figure \ref{fig01}. For each graph $G_1$ and $G_2$ we construct the corresponding bi-stochastic graph Laplacians 
$\mathbf{L}_1$ and $\mathbf{L}_2$. Next we solve the optimization problem
$$
\mathbf{t}^* = \argmax_{\mathbf{t} \in T}
\lambda_1(\mathbf{L}_{\mathbf{t}}),
$$
where $\mathbf{L}_\mathbf{t}$ is defined in \eqref{Lteq}.
By setting $t_2 := 1-t_1$ we can optimize over $t_1 \in [0,1]$. 
To visualize this optimization problem, we plot 
$\lambda_1(\mathbf{L}_{\mathbf{t}})$ against $t_1$ in Figure \ref{fig02}.
\begin{figure}[h!]
\includegraphics[width=.35\textwidth]{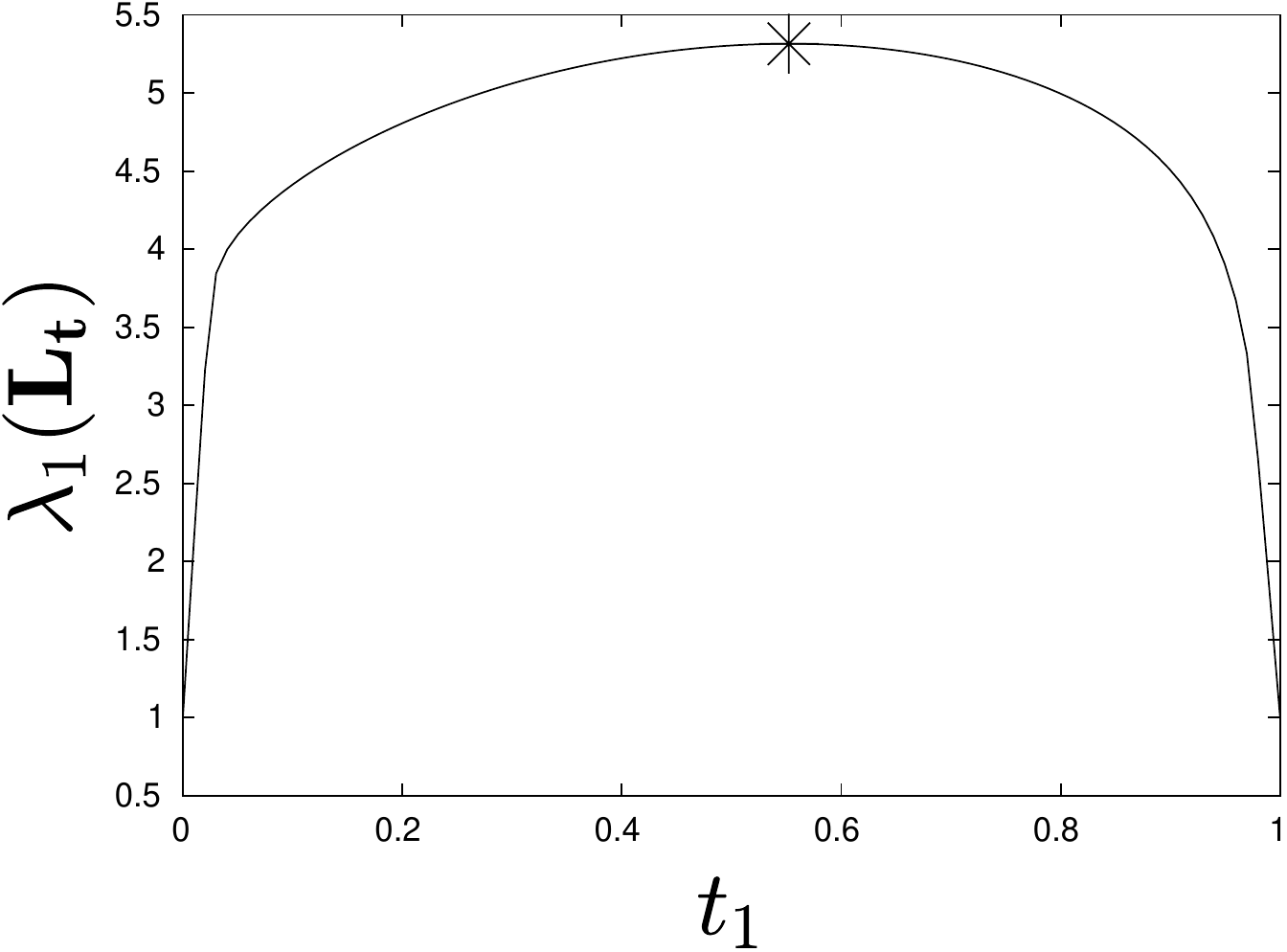} 
\caption{Parameter $t_1$ versus $\lambda_1(\mathbf{L}_{\mathbf{t}})$; the max
occurs at the star.}
\label{fig02}
\end{figure}

\noindent Using numerical optimization, we find that
$$
\mathbf{t}^* \approx (0.552330195903778,   0.447669804096222).
$$
Next, we use Theorem \ref{thm1} to validate the results of the optimization, which gives
$$
\left| \max_{k \in \{1,2\}} s_{\mathbf{L}_k}
(\boldsymbol{\psi}_1(\mathbf{L}_{\mathbf{t}^*})) - \lambda_1(\mathbf{L}_{\mathbf{t}^*})
\right|\le  8.802488427051003 \times 10^{-8},
$$
indicating that the optimization procedure was successful. Since we know how the graphs $G_1$ and $G_2$ were 
generated, we can further validate the method by checking that 
$\boldsymbol{\psi}_1(\mathbf{L}_{\mathbf{t}^*})$ is a smooth function
of the common variable that influences edge creation in both graphs (the
distance of a point  from the origin). We plot
$\boldsymbol{\psi}_1(\mathbf{L}_{\mathbf{t}^*})$ versus $r$ (representing the
distance of a point from the origin) in Figure \ref{fig03}.
\begin{figure}[h!]
\begin{tabular}{ccc}
\includegraphics[width=.28\textwidth]{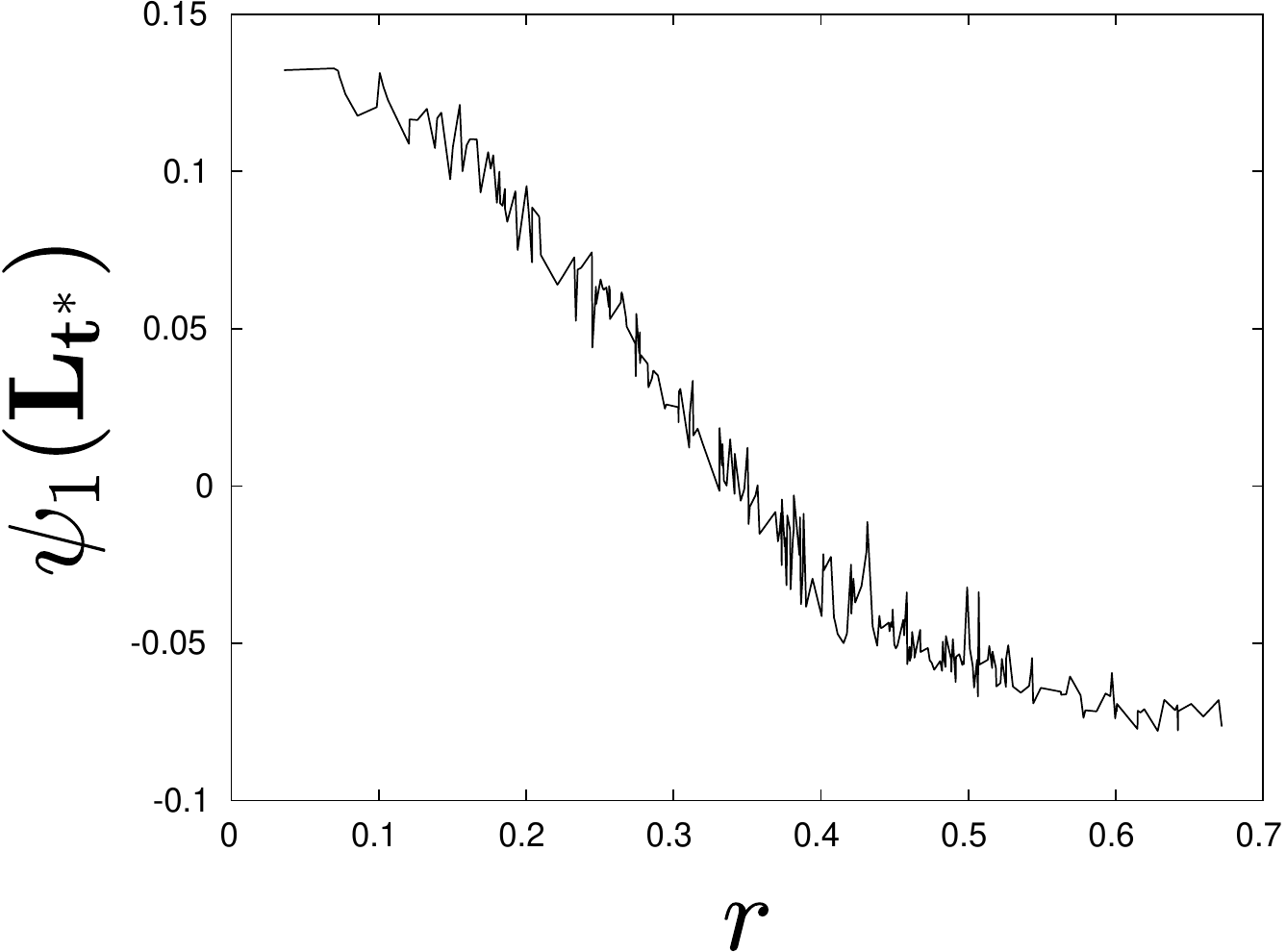} &
\includegraphics[width=.28\textwidth]{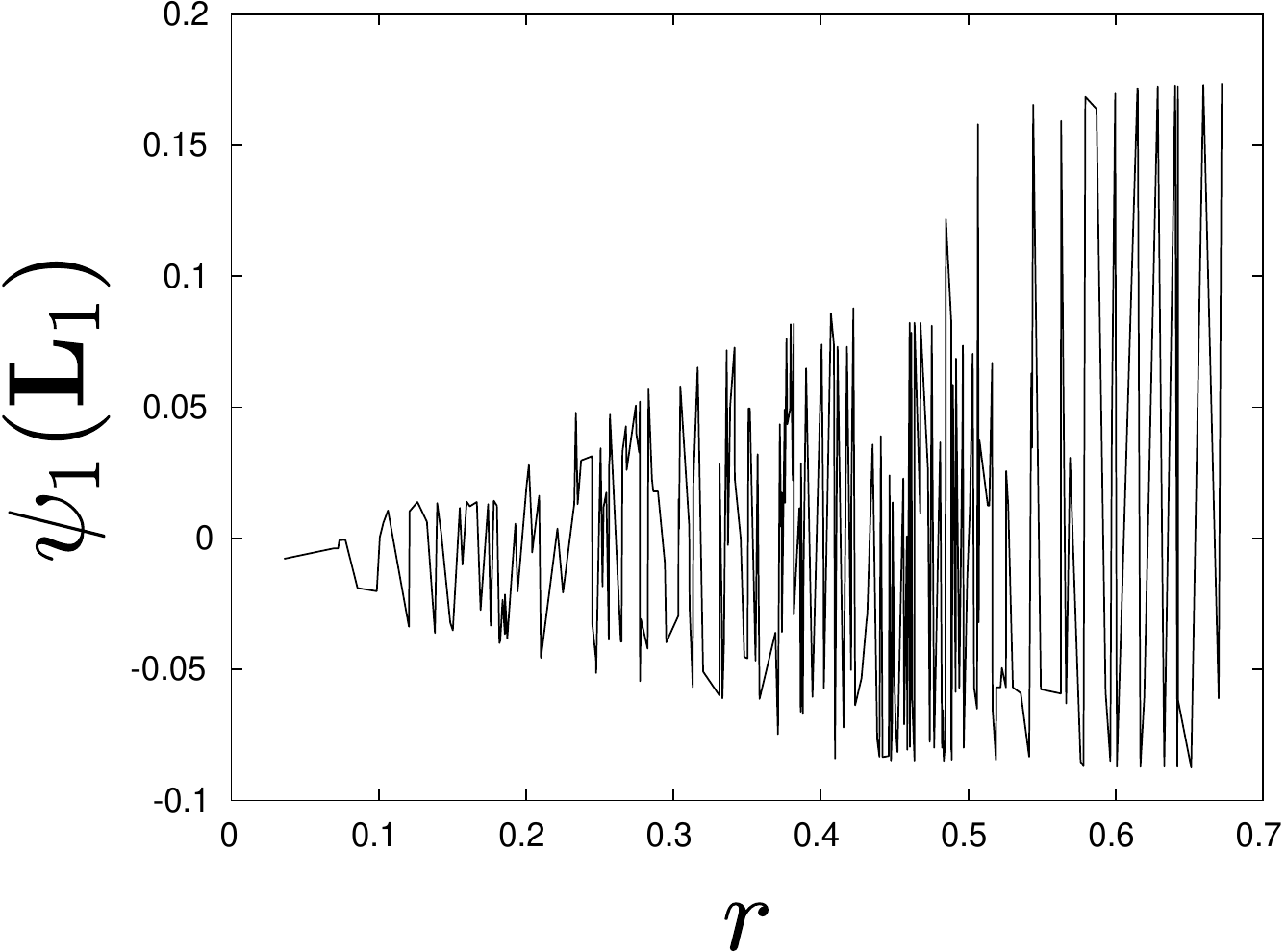} &
\includegraphics[width=.28\textwidth]{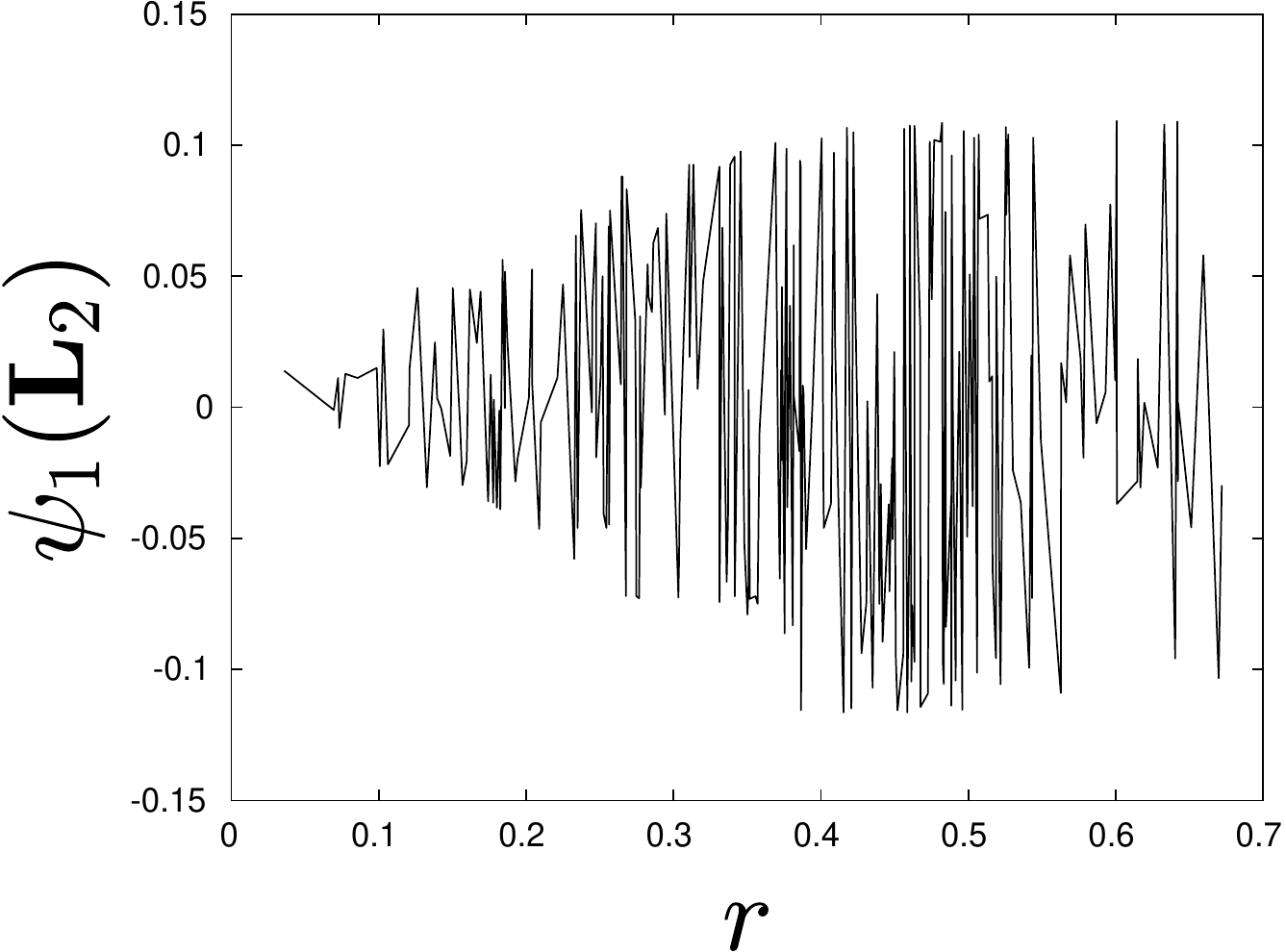} 
\label{fig03}
\end{tabular}
\caption{Plots of 
$\boldsymbol{\psi}(\mathbf{L}_{\mathbf{t}^*})$, $\boldsymbol{\psi}(\mathbf{L}_1)$, and $\boldsymbol{\psi}(\mathbf{L}_2)$ against $r$.}
\end{figure}

Observe that in Figure \ref{fig03} the common variable is essentially a
re-scaling of the distance to
the origin (as would be expected). 
For comparison, Figure \ref{fig03} also includes plots of $\boldsymbol{\psi}_1(\mathbf{L}_{1})$ and
$\boldsymbol{\psi}_1(\mathbf{L}_{2})$ versus $r$ to
demonstrate that neither of  them are smooth with respect to the common variable.

In the following section, we will present a similar example of building graphs from randomly 
rotated points except we start with points in three dimensions, and perform 
rotations around different axes to demonstrate how the method works when there are three 
graphs $G_1$, $G_2$, and $G_3$.

\subsection{Independent rotations in three dimensions}

Let $X_1= \{x_1,\ldots,x_n\}$ be $n=500$ independent uniformly
random points from the unit ball $\{x \in \mathbb{R}^3 : \|x\|_{\ell^2} \le 1\}$, and
let $\theta_1,\ldots,\theta_n$ and $\phi_1,\ldots,\phi_n$ be independent uniformly 
random angles from $[0,2\pi)$. Set
$$
X_2 = \{ T_{\theta_1}(x_1),\ldots,T_{\theta_n}(x_n)\},
$$
where $T_{\theta_j}(x)$ is a rotation about the $z$-axis by angle $\theta$: if
$x = (r \cos \theta, r \sin \theta, z)$, then
$$
T_{\theta_j}(x) = (r \cos(\theta+\theta_j), r \sin(\theta+\theta_j), z),
$$
and set
$$
X_3 = \{ S_{\phi_1}(x_1),\ldots,S_{\phi_n}(x_n)\},
$$
where $S_{\phi_j}(x)$ is a rotation about the $y$-axis: 
if $x = (r\cos \phi, y, r \sin \phi)$, then
$$
S_{\phi}(x) = (r \cos(\phi+\phi_j), y, r \sin(\phi+\phi_j)).
$$
We construct  $6$-nearest neighbor graphs $G_1$, $G_2$ and $G_3$ from the sets
 $X_1$, $X_2$, and $X_3$, respectively, see Figure \ref{fig04}.

\begin{figure}[h!]
\centering
\begin{tabular}{ccc}
\includegraphics[width=.3\textwidth]{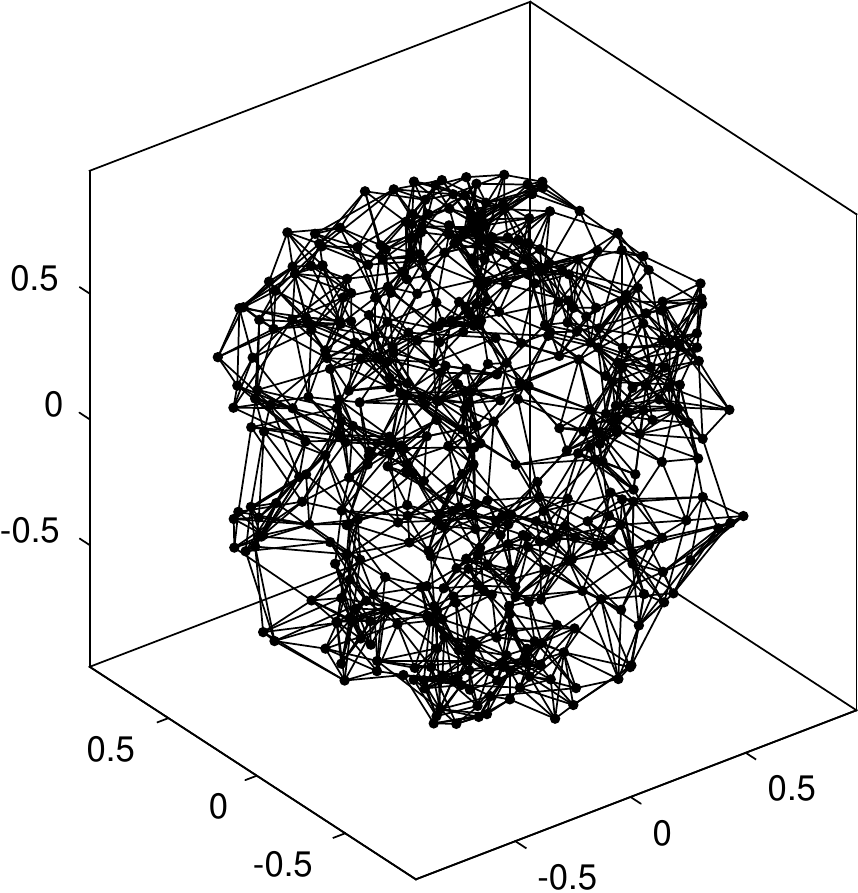} &
\includegraphics[width=.3\textwidth]{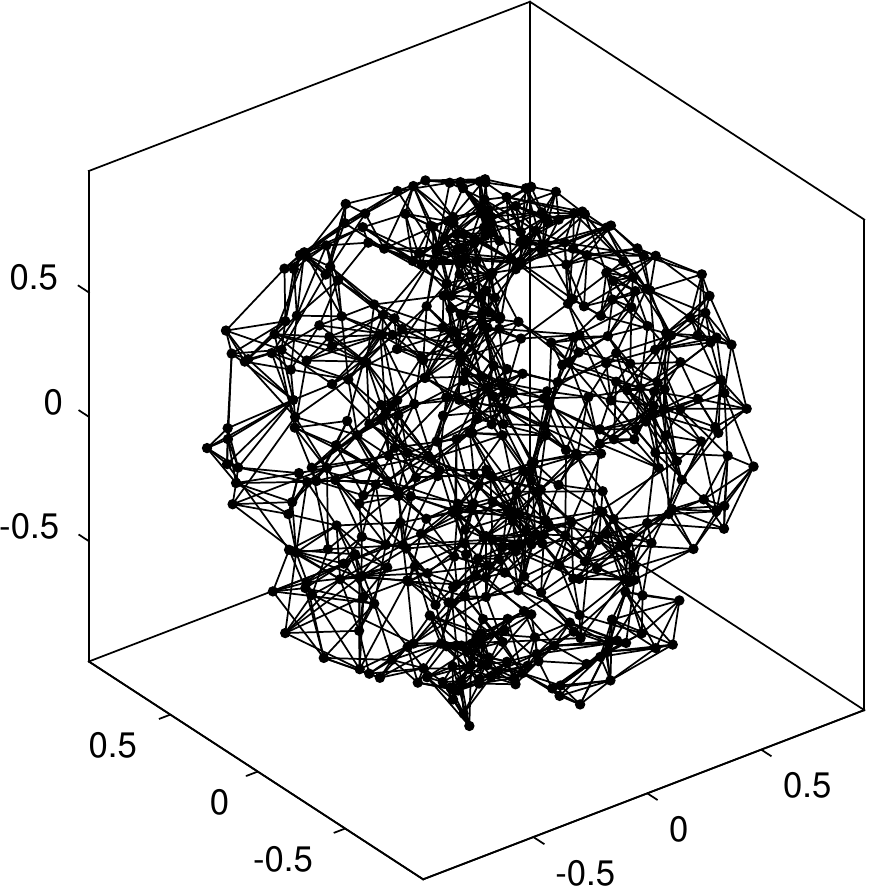} &
\includegraphics[width=.3\textwidth]{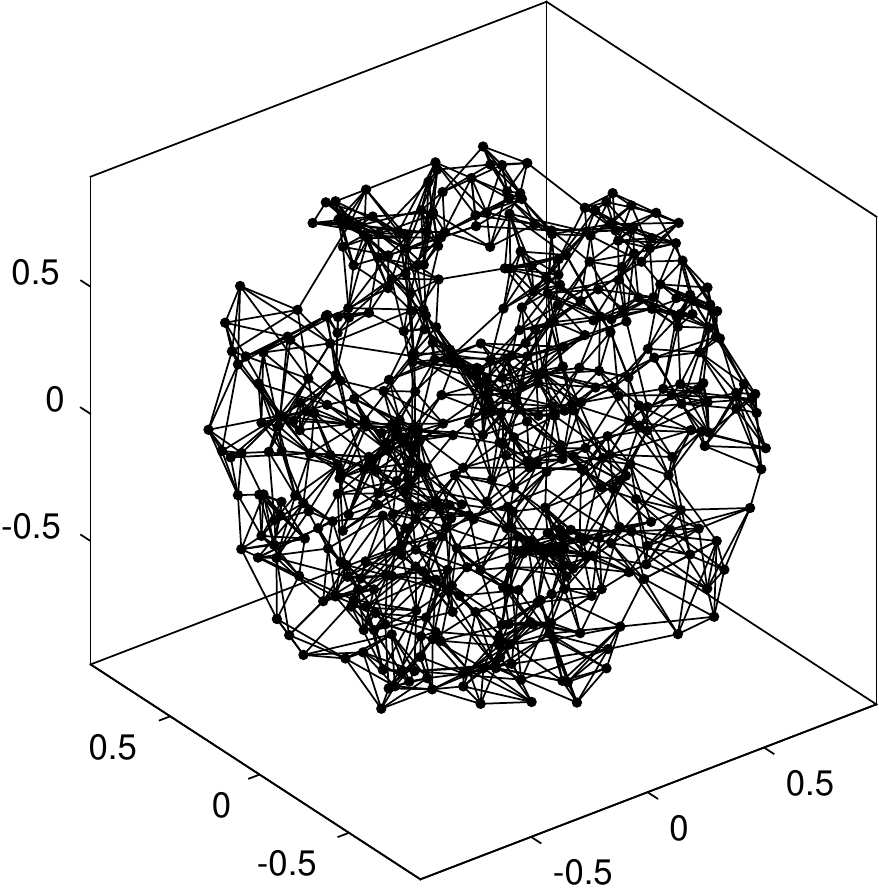} 
\end{tabular}
\caption{The graphs $G_1$ (left), $G_2$ (middle), and $G_3$ (right).}
\label{fig04}
\end{figure}

\noindent For each graph $G_1$, $G_2$ and $G_3$ we construct the corresponding bi-stochastic graph Laplacians 
$\mathbf{L}_1$, $\mathbf{L}_2$, and $\mathbf{L}_3$ and consider the optimization problem
$$
\mathbf{t}^* = \argmax_{\mathbf{t} \in T}
\lambda_1(\mathbf{L}_{\mathbf{t}}).
$$
By setting $t_3 = 1 - (t_1 + t_2)$ we can optimize $\lambda_1(\mathbf{L}_{\mathbf{t}})$
over $(t_1,t_2)$ such that $0 \le t_1,t_2$ and $t_1 + t_2 \le 1$, see Figure \ref{fig05}.
\begin{figure}[h!]
\centering
\includegraphics[width=.6\textwidth]{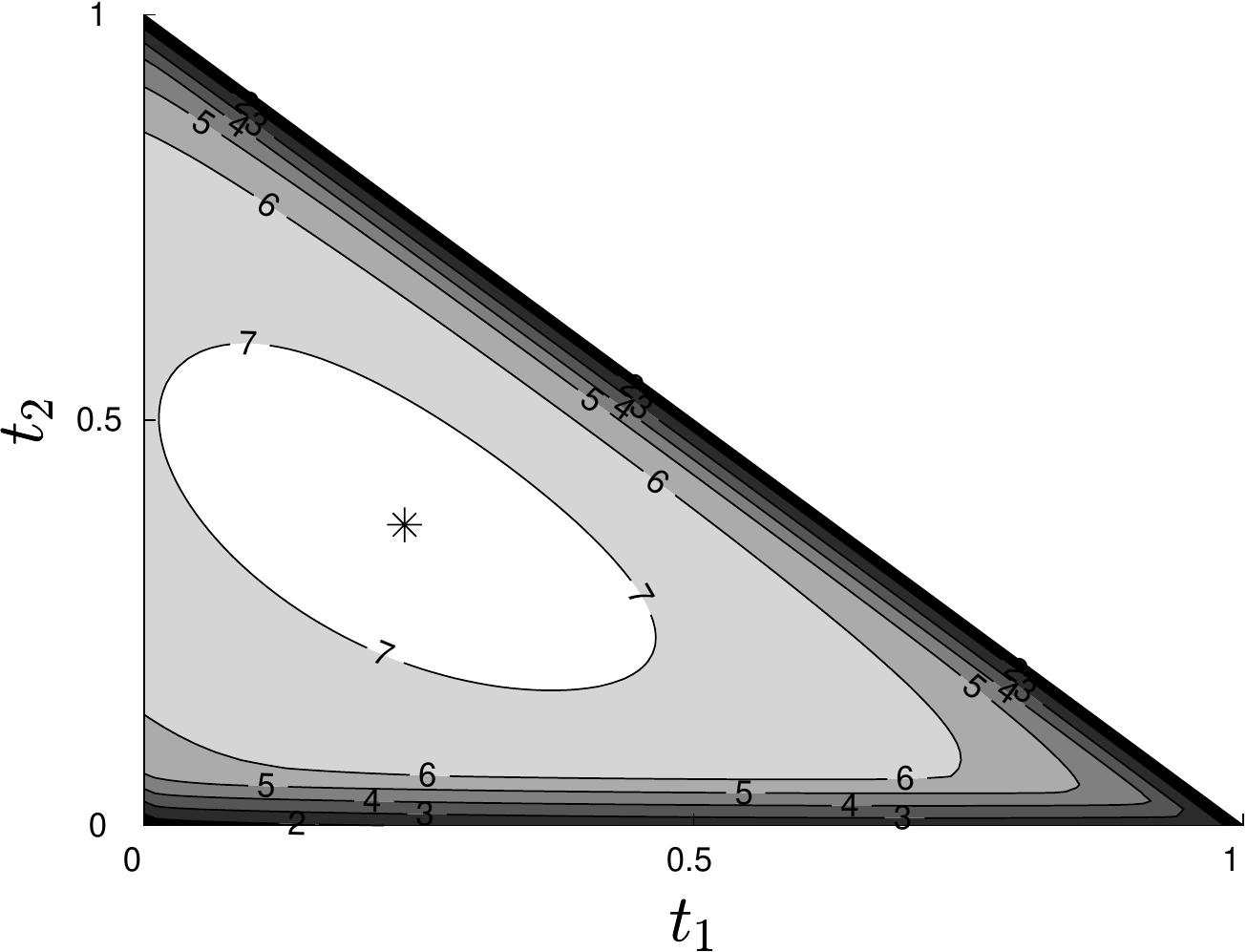} 
\caption{A contour plot of $\lambda_1(\mathbf{L}_{\mathbf{t}})$ for
 $0 \le t_1,t_2$ and $t_1 + t_2 \le 1$.}
\label{fig05}
\end{figure}
Using numerical optimization, we find that
$$
\mathbf{t}^* \approx (0.236853469652210,  0.371066650569015,  0.392079879778775).
$$
Validating the results of this optimization procedure using Theorem \ref{thm1} gives
$$
\left| \max_{k \in \{1,2,3\}} s_{\mathbf{L}_k}(\boldsymbol{\psi}_1(\mathbf{L}_{\mathbf{t}^*})) - \lambda_1(\mathbf{L}_{\mathbf{t}^*})
\right|\le   9.502285891471729 \times 10^{-8},
$$
so the numerical results are very close to optimal. Since we know how the 
graphs were constructed, we can further interpret the result. As in the previous example, 
the common variable is the distance of a point to the origin. To demonstrate that
 $\boldsymbol{\psi}_1(\mathbf{L}_{\mathbf{t}^*})$ is a re-scaling of the common variable,
 we plot $\boldsymbol{\psi}_1(\mathbf{L}_{\mathbf{t}^*})$ versus the distance to the 
 origin $r$; for comparison, we also plot  $\boldsymbol{\psi}_1(\mathbf{L}_{1})$, 
 $\boldsymbol{\psi}_1(\mathbf{L}_{2})$,
 and  $\boldsymbol{\psi}_1(\mathbf{L}_{3})$ against $r$, see Figure \ref{fig06}.

\begin{figure}[h!]
\begin{tabular}{cccc}
\includegraphics[width=.22\textwidth]{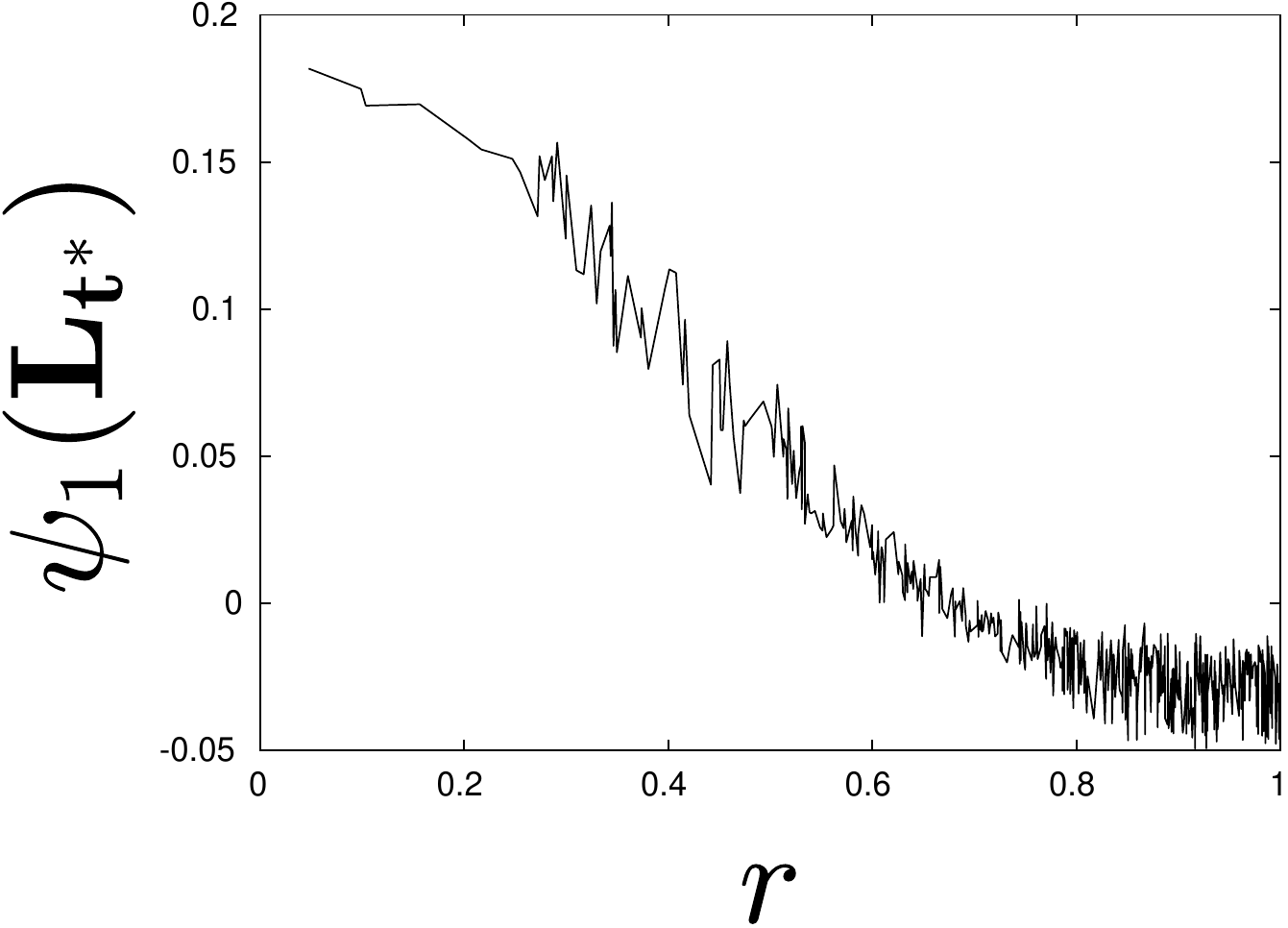} &
\includegraphics[width=.22\textwidth]{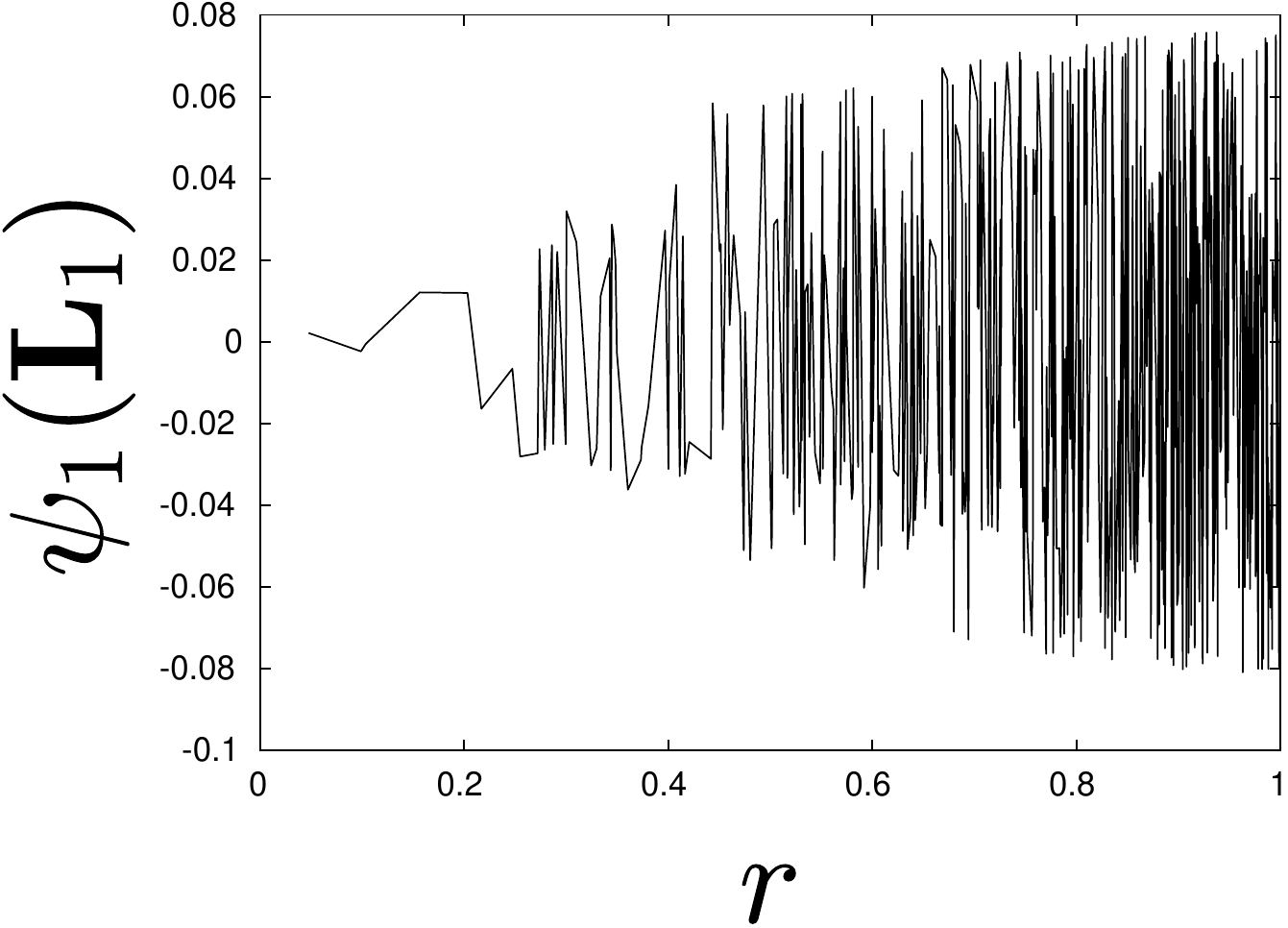} &
\includegraphics[width=.22\textwidth]{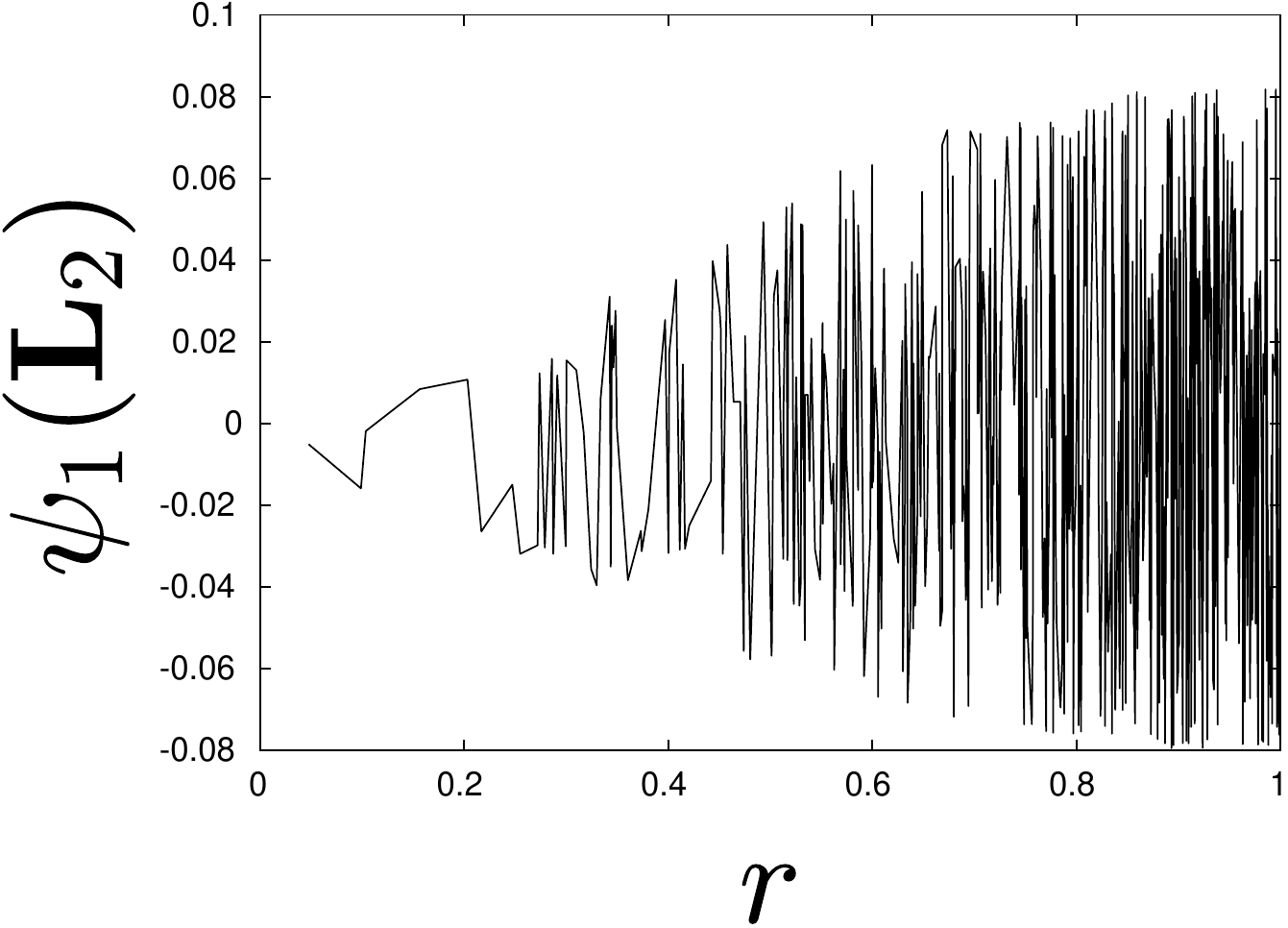}  &
\includegraphics[width=.22\textwidth]{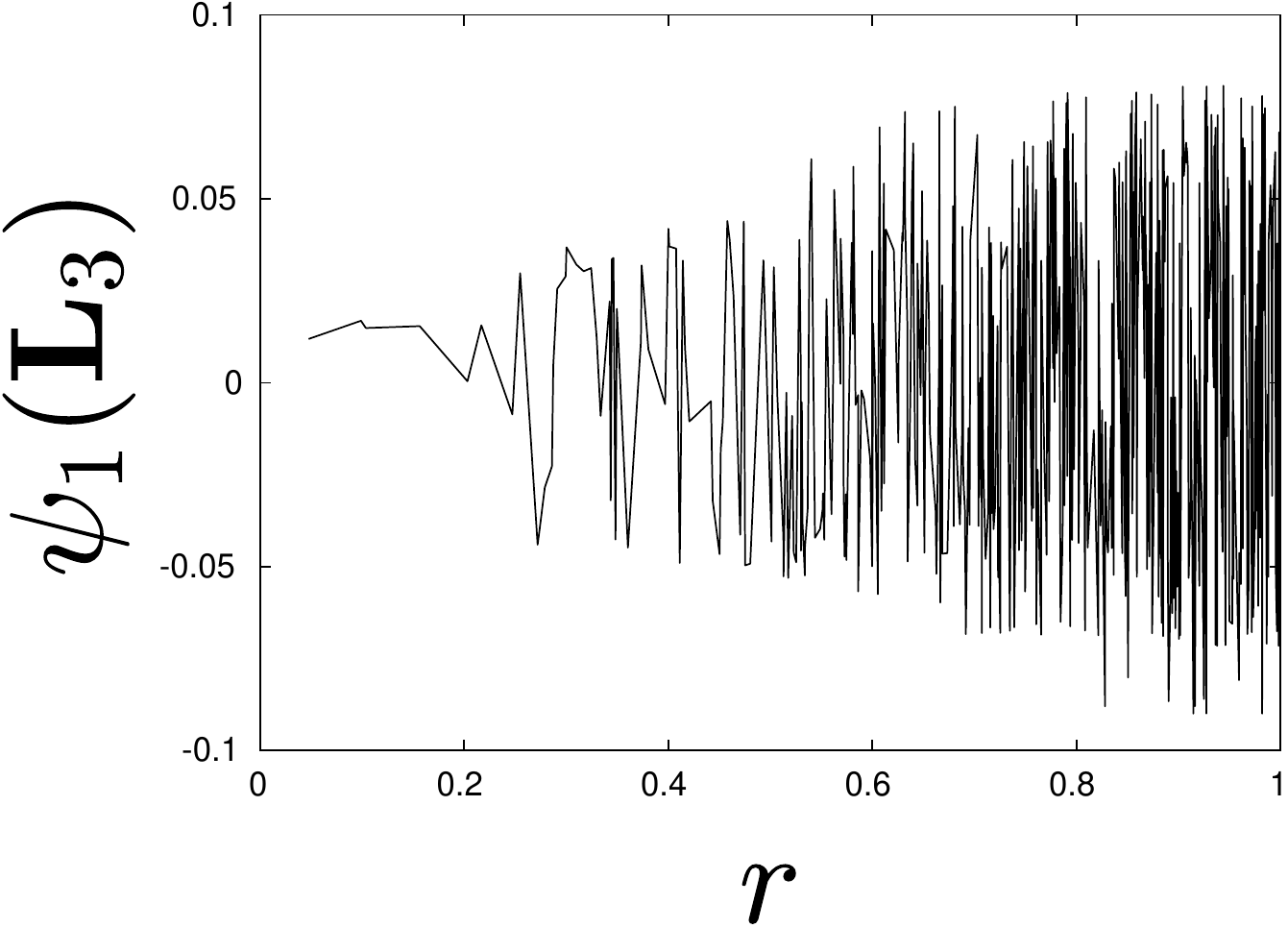} 
\end{tabular}
\caption{The first nontrivial eigenvectors of $\mathbf{L}_{\mathbf{t}^*}$,
$\mathbf{L}_1$, $\mathbf{L}_2$, and $\mathbf{L}_3$ versus 
the common variable $r$.}
\label{fig06}
\end{figure}

\noindent Finally, we note that this example has some interesting asymmetry. In Figure \ref{fig05} 
observe that the level line $\lambda_1(\mathbf{L}_\mathbf{t}) = 7$ (the closest 
level line to the maximum value) almost intersects the line $t_1 =0$.
  In contrast, the value of 
$\lambda_1(\mathbf{L}_{\mathbf{t}})$ on the lines $t_2 = 0$ and $t_3 = 1-(t_1+t_2)$ 
are close to $1$. This indicates that just using
the graphs $\{G_2,G_3\}$ could allow us to approximately determine the common variable, 
while using $\{G_1,G_2\}$ or $\{G_1,G_3\}$ would give bad results. Why is this the case?
By definition points in $X_1,X_2$ have the same $z$-coordinate, and points in
$X_1,X_3$ have the same $y$-coordinate, while the only common variable for points 
in $X_2,X_3$ is the distance of a point from the origin. For example, if
we just consider $X_1,X_2$, then the function $f(x,y,z) = z$ is smooth with
respect to $G_1$ and $G_2$, but not smooth with respect to $G_3$.

\subsection{Horizontal and vertical barbell example}
Next, we provide a degenerate example, where we are given three graphs $G_1$, $G_2$, 
$G_3$, and the optimal value of $\mathbf{t} = (t_1,t_2,t_3)$ occurs on the boundary of 
the region $\{ (t_1,t_2) : 0 \le t_1,t_2 \text{ and }t_1 + t_2 \le 1\}$.
Let $D = \{ x \in \mathbb{R}^2 : \|x\|_{\ell^2}\le 1\}$ be the unit disc, and
define the functions $f,g : \mathbb{R}^2 \rightarrow \mathbb{R}^2$ by 
$$
f(x,y) = \big(x,y \cdot(1-\cos \pi x)\big), \quad \text{and} \quad g(x,y) =
\big(x \cdot (1-\cos \pi y),y\big). 
$$
Informally speaking, the maps $f$ and $g$ squeeze the disc into a horizontal
barbell shape and a vertical barbell shape, respectively.
Let $X_1= \{x_1,\ldots,x_n\}$ be a set of $n=250$ independent uniformly random
points from the unit disc $D$. Set $X_2 = f(X_1)$, and $X_3 = g(X_1)$, and let
$G_1,G_2,G_3$ be  $6$-nearest neighbor graphs of $X_1,X_2,X_3$, respectively,
see Figure \ref{fig07}. 
\begin{figure}[h!]
\centering
\begin{tabular}{ccc}
\includegraphics[width=.3\textwidth]{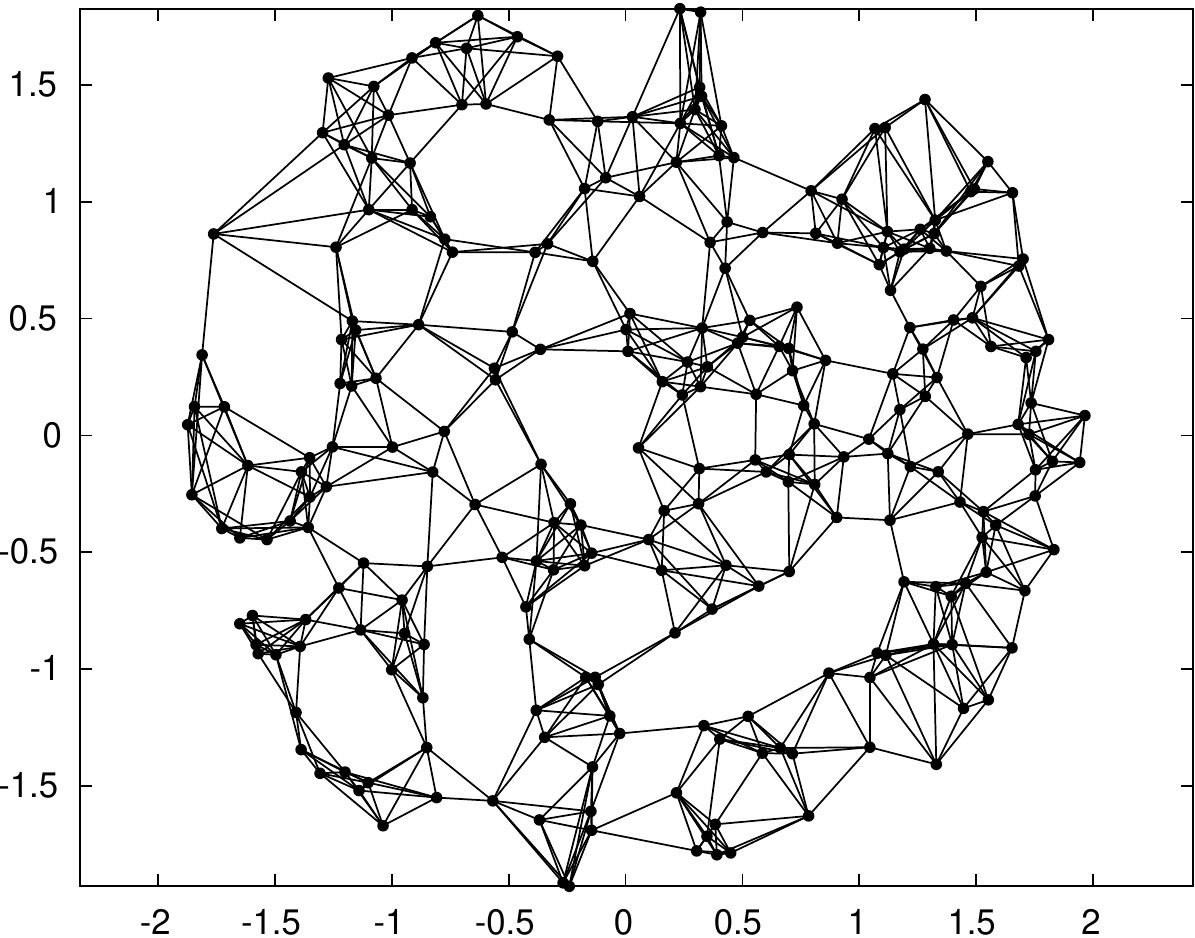} &
\includegraphics[width=.3\textwidth]{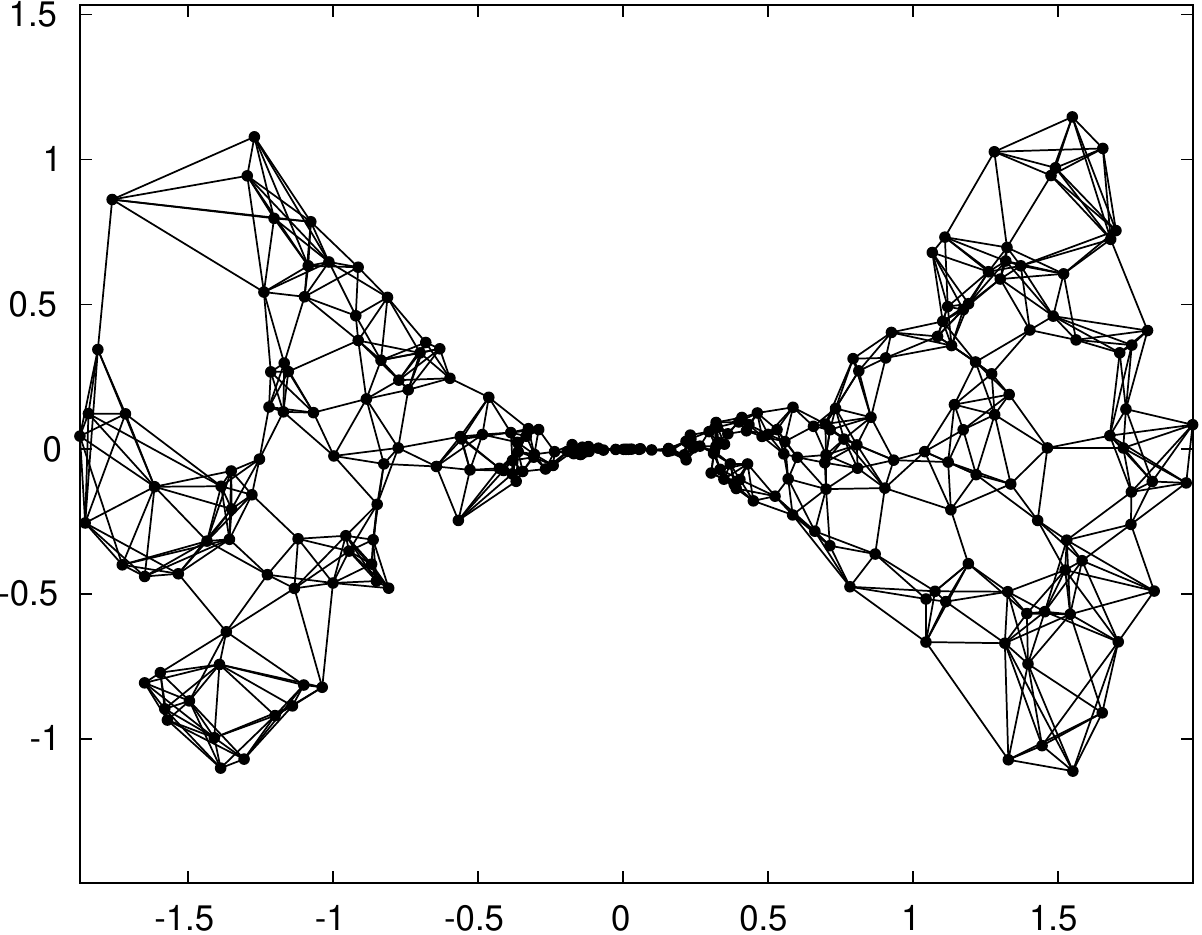} &
\includegraphics[width=.3\textwidth]{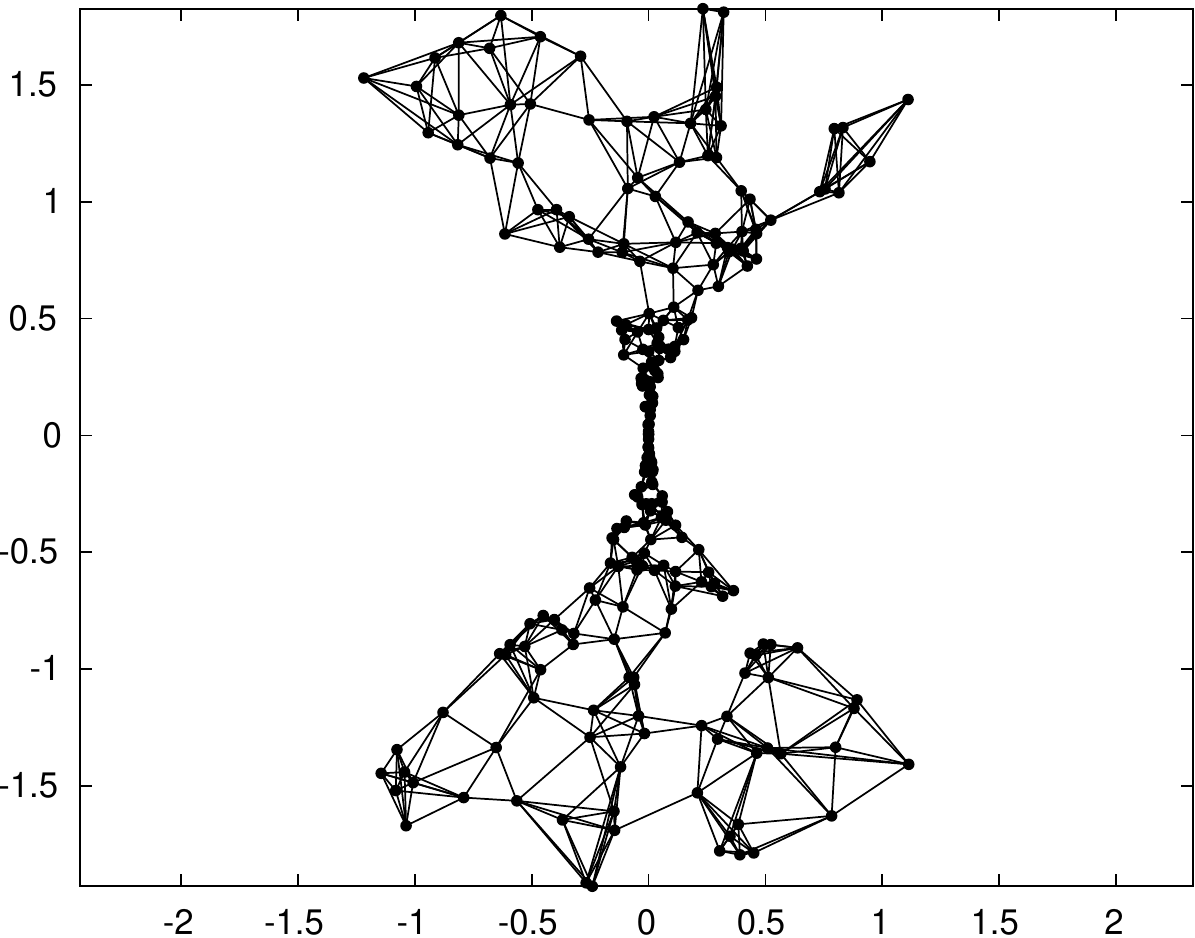} \\
\end{tabular}
\caption{The graphs $G_1$ (left), $G_2$ (middle), and $G_3$
(right).}
\label{fig07}
\end{figure}

Let $\mathbf{L}_1$, $\mathbf{L}_2$, and $\mathbf{L}_3$ be the bi-stochastic graph Laplacians of 
$G_1$, $G_2$, and $G_3$, respectively. 
We plot $\lambda_1(\mathbf{L}_\mathbf{t})$ versus $(t_1,t_2)$ in Figure \ref{fig08}.
\begin{figure}[h!]
\centering
\includegraphics[width=.6\textwidth]{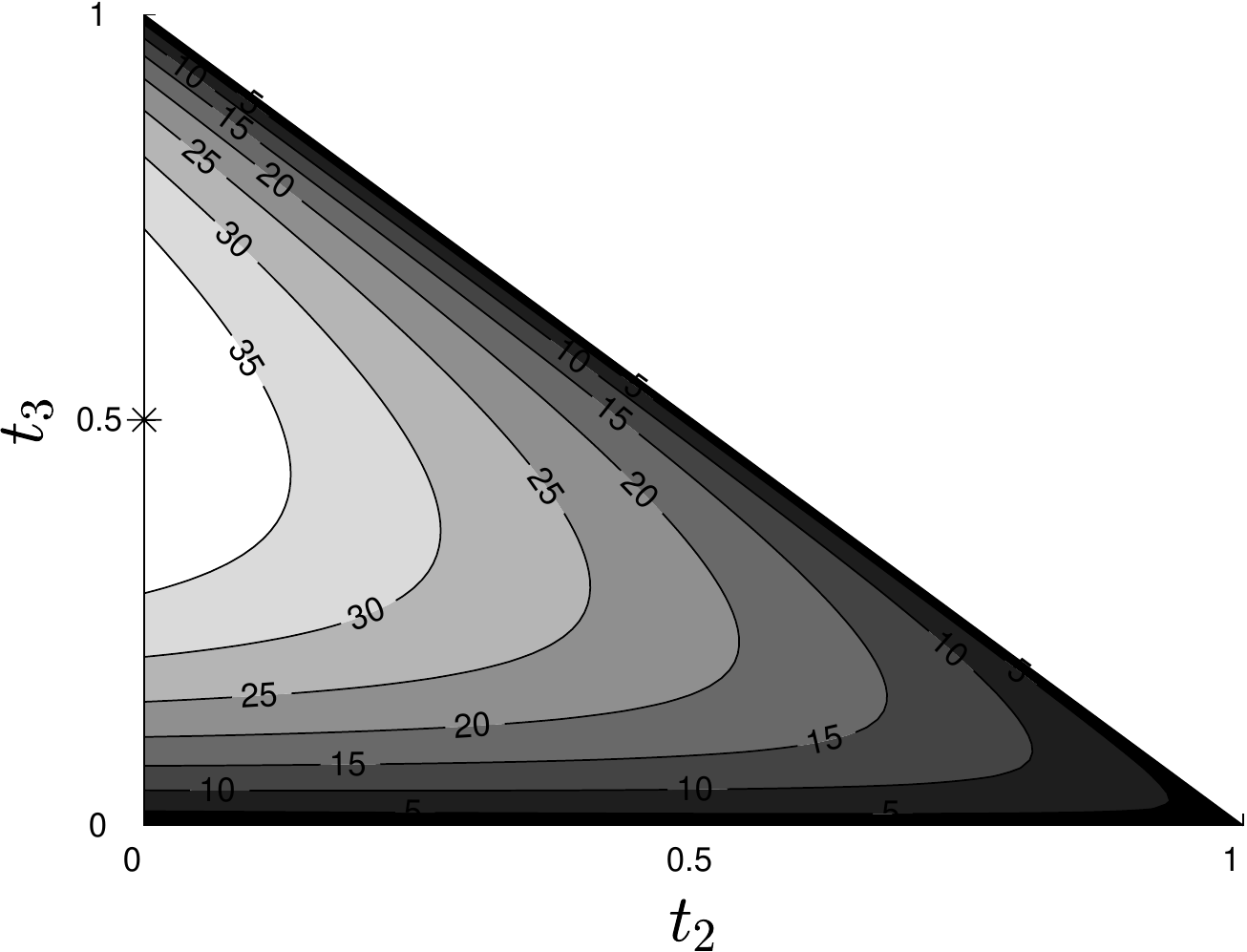} 
\caption{Contour plot of $\lambda_1(\mathbf{L}_\mathbf{t})$ for $(t_1,t_2)$ such that 
$0 \le t_1,t_2$ 
and $t_1 + t_2 \le 1$. The maximum is attained at the star.} 
\label{fig08}
\end{figure}

Using numerical optimization we find that
$$
\mathbf{t}^* \approx (2.871019259460022 \times 10^{-15}, .5005153871788890,    .4994846128211081);
$$
and using Theorem \ref{thm1} to compute an error estimate gives
$$
\left|\max_{k \in \{1,2,3\}} s_{\mathbf{L}_k}(\boldsymbol{\psi}_1(\mathbf{L}_{\mathbf{t}^*})) - \lambda_1(\mathbf{L}_{\mathbf{t}^*})
\right| \le  2.807697995876879 \times 10^{-6},
$$
which verifies that we have solved the optimization problem correctly. 
Interestingly, the optimal value occurs on the boundary on of $\{(t_1,t_2) : 
0 \le t_1,t_2 \text{ and } t_1 + t_2 \le 1 \}$. Since we know how the graphs were created, 
this behavior makes sense: the sets $X_2$ and $X_3$ are modifications of $X_1$ where 
points have been squeezed together, which makes the corresponding vertices highly 
connected in the graphs $G_2$ and $G_3$. This in turn imposes extra conditions for a 
function to be smooth with respect to $G_2$ or $G_3$. Furthermore, most of the edges 
appearing in $G_1$ appear either in $G_2$ or $G_3$.

For this example, the common variable is less straightforward to define. However, one
property that is maintained under the deformation is as follows: points in the same 
quadrant of the plane should remain connected across all graphs. In particular, we can 
partition the points in $X_1$ into four groups 
$$
\begin{array}{ccc}
NE &=& \{(x,y) \in X_1 : x \ge 0, y \ge 0\}, \\
NW &=& \{(x,y) \in X_1 : x < 0, y > 0\}, \\
SW &=& \{(x,y) \in X_1 : x < 0, y < 0\}, \\
SE &=& \{(x,y) \in X_1 : x \ge 0, y < 0\}.
\end{array}
$$
To understand what the optimal Laplacian $\mathbf{L}_{\mathbf{t}^*}$ is encoding, we 
plot the first three (nontrivial eigenvectors) of this operator, see Figure \ref{fig09}.

\begin{figure}[h!]
\centering
\includegraphics[width=.6\textwidth]{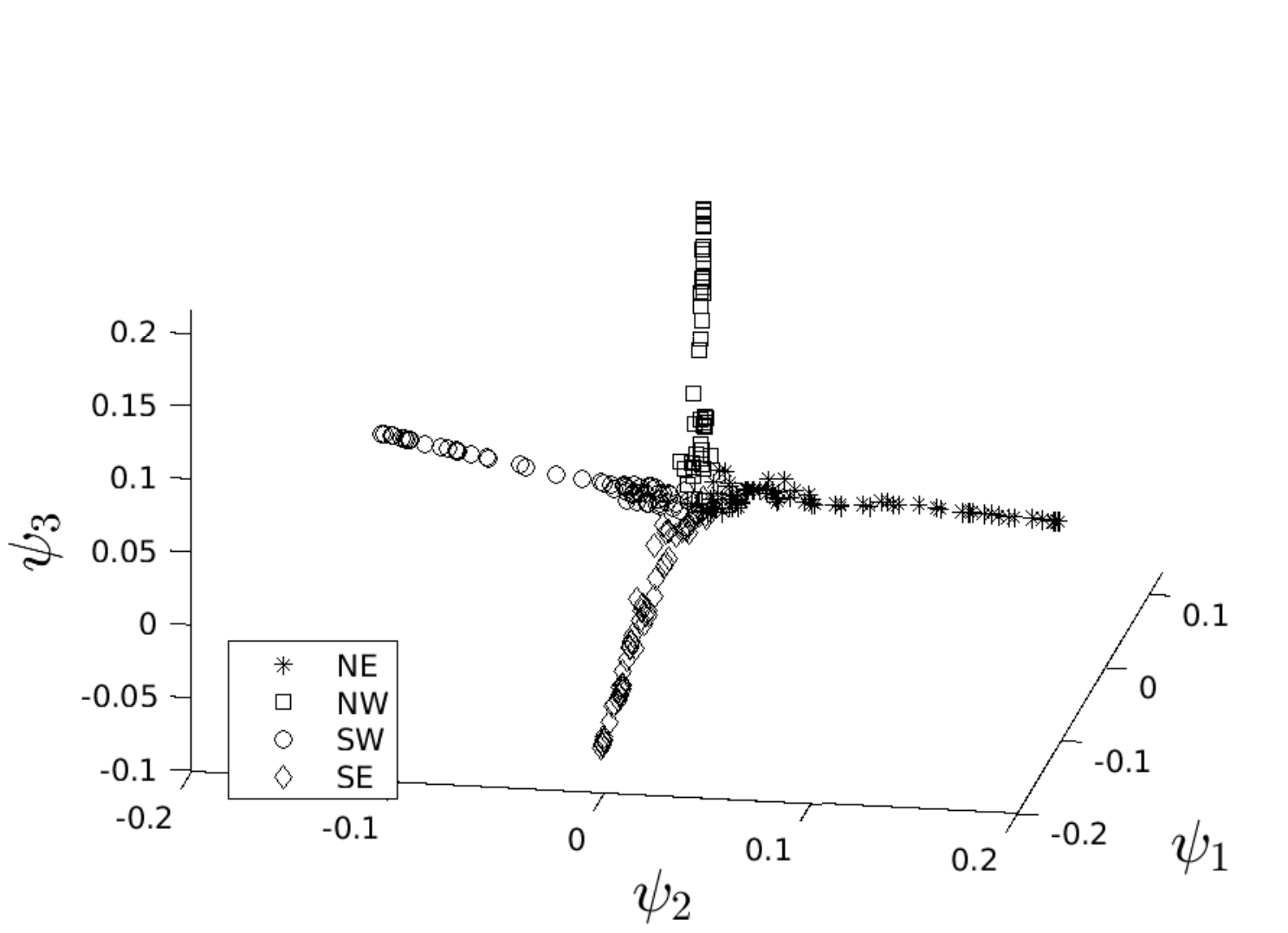} 
\caption{The first three nontrivial eigenvectors of $\mathbf{L}_{\mathbf{t}^*}$ with markers
indicating the quadrant of the points in $X_1$}
\label{fig09}
\end{figure}

\noindent Running $k$-means clustering on the
embedding in Figure \ref{fig09} would approximately recover the different
groups of points NE, NW, SW, and SE.

\begin{remark}[Common information spectral clustering]
Spectral clustering is a clustering method whose first step is to embed the given data points
using the eigenvectors of an operator followed by running the $k$-means clustering
algorithm. The method in this paper
can be used to perform a common information spectral clustering algorithm by
using the eigenvectors $\mathbf{L}_{\mathbf{t}^*}$ to embed the points, 
and then running $k$-means clustering. Running $k$-means on Figure \ref{fig09} is an example of this common information spectral clustering.
\end{remark}


\subsection{Spiral and Torus} \label{spiraltorus} We conclude with an example illustrating Remark 
\ref{othernorm}: it can be advantageous to change the notion of smoothness. The two graphs in this example are a spiral in the plane and a two-dimensional torus embedded in $\mathbb{R}^3$ (see Figure \ref{fig10}). Formally, let $\{\theta_1,\ldots,\theta_n\}$ and $\{\phi_1,\ldots,\phi_n\}$ be $n=500$ independent
uniformly random angles from $[0,2\pi)$, and let $\{t_1,\ldots,t_n\}$ be independent
uniformly random points from the interval $[.25,1.5)$. We define the Spiral set
$X_1 = \{x_{1,1},\ldots,x_{n,1}\} \subset \mathbb{R}^2$ by
$$
x_{1,j} = (t_j + 0.45 \phi_j/(2 \pi)) \big( \cos(4 \pi (t_j-.25)/1.5), 
 \sin(4 \pi (t_j-.25)/1.5) \big).
$$
and define the Torus set $X_2 = \{x_{2,1},\ldots,x_{2,n}\} \subset \mathbb{R}^3$
by
$$
x_{2,j} = \big( (.75+.25\cos(\phi_j)) \cos(\theta_j),
(.75+.25 \cos(\phi_j))\sin(\theta_j), .25 \sin(\phi_j) \big),
$$
For each set $X_1$ and $X_2$ we construct $6$-nearest neighbor graphs $G_1$ and $G_2$,
see Figure \ref{fig10}.
\begin{figure}[h!]
\begin{tabular}{cc}
\includegraphics[width=.4\textwidth]{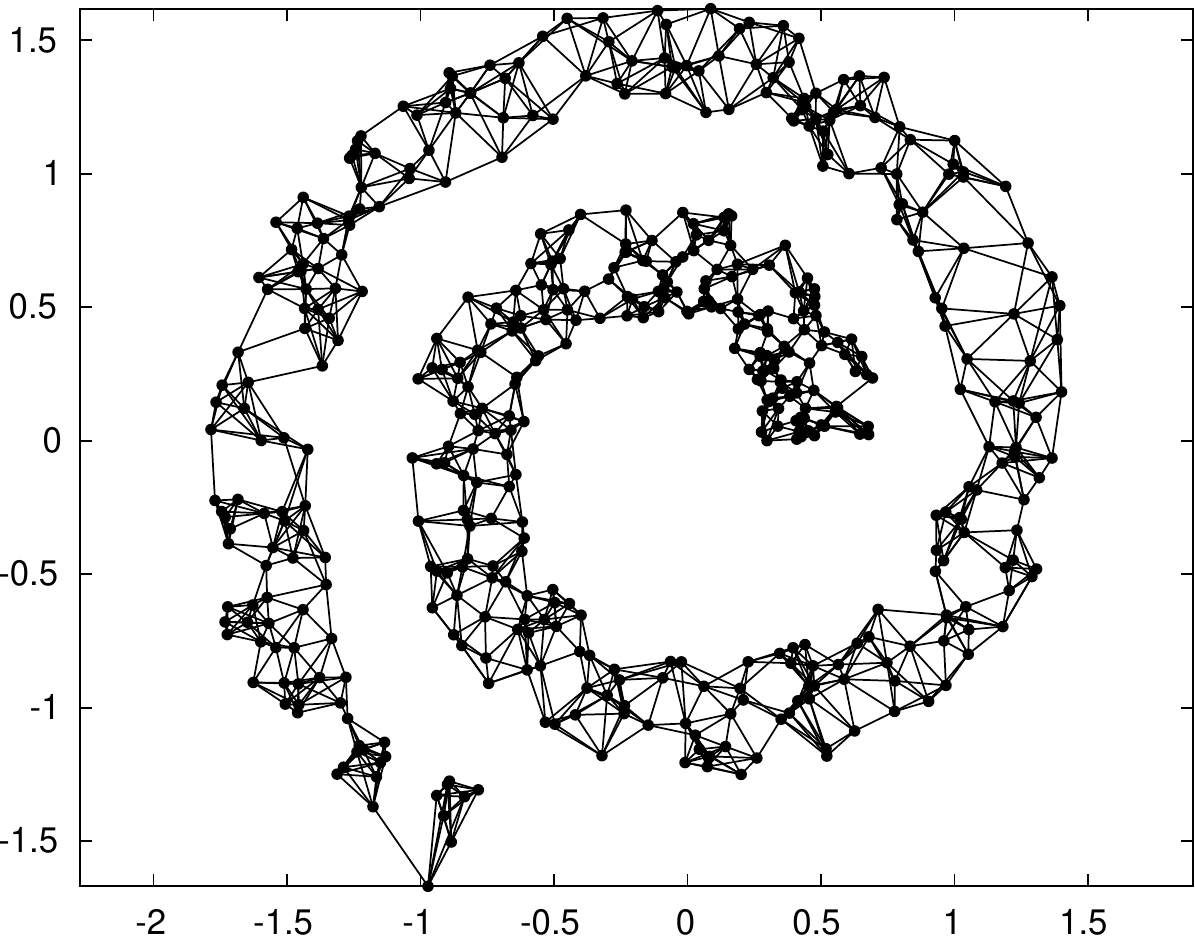} &
\includegraphics[width=.5\textwidth]{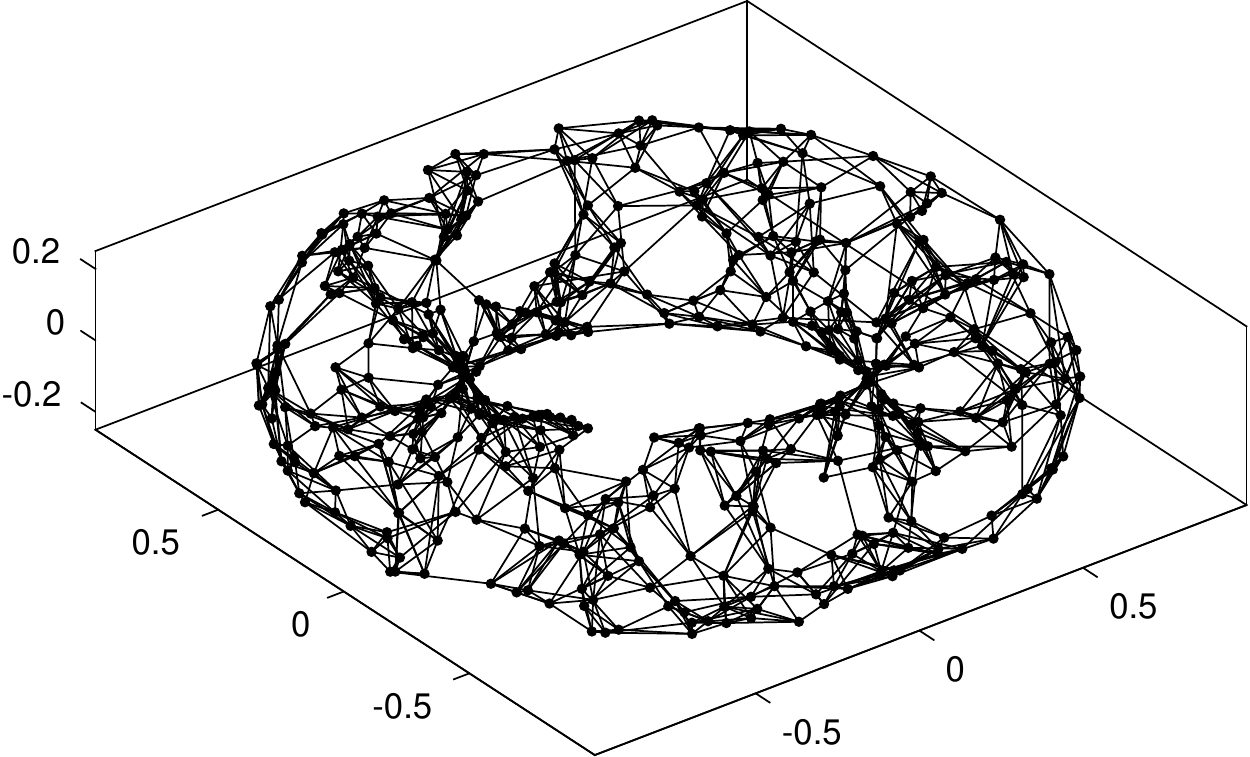} 
\end{tabular}
\caption{The graph $G_1$ (left), and graph $G_2$ (right).} \label{fig10}
\end{figure}

\noindent The common variable used to define both graphs is the parameter
$\phi_j$. The parameter $\phi_j$ controls the location in the 
width of the spiral (the width is very small compared to the height), and similarly, 
controls the location of a point along the smaller circle used to form the torus. The definition
\begin{equation} \label{smoothnessex}
s_{\mathbf{L}}(\mathbf{x}) = \frac{1}{\lambda_1(\mathbf{L})}
\mathbf{x}^\top \mathbf{L} \mathbf{x}
\end{equation}
has, in this example, a significant downside: the spiral is only weakly connected and has a very small first Laplacian eigenvalue. The renormalization ensures that $s_{\mathbf{L}}(\mathbf{x})$ is 1, when $\mathbf{x}$ is the first Laplacian eigenvector, however, it will be exceedingly large for all vectors in the orthogonal complement. The degeneracy in the spectrum implies that  $s_{\mathbf{L}}(\mathbf{x})$ simply does not accurately capture the spectral geometry of the spiral. The normalization 
from Remark \ref{othernorm} 
\begin{equation} \label{smoothnessav}
a_{\mathbf{L}}(\mathbf{x}) = 
\left( \frac{1}{n-1} \sum_{j=1}^{n-1} \lambda_j(\mathbf{L}) \right)^{-1}
\mathbf{x}^\top \mathbf{L} \mathbf{x}
\end{equation}
provides a reasonable alternative: we keep the quadratic form 
$\mathbf{x}^\top \mathbf{L} \mathbf{x}$ but use a normalization which maintains the global structure of the spectrum better. This is also illustrated in
Figure \ref{fig11}.

\begin{figure}[h!]
\centering
\begin{tabular}{cc}
\includegraphics[width=.4\textwidth]{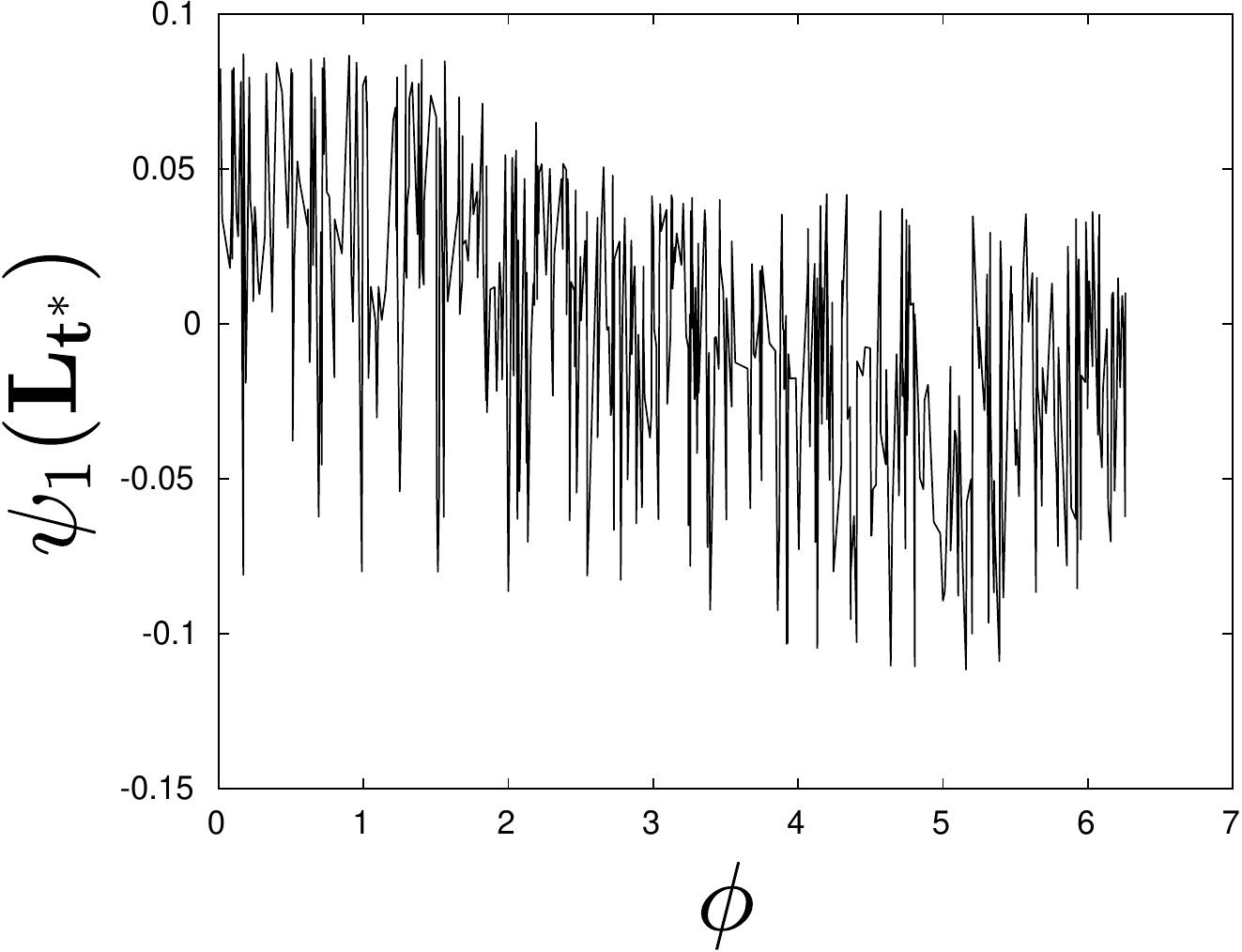} &
\includegraphics[width=.4\textwidth]{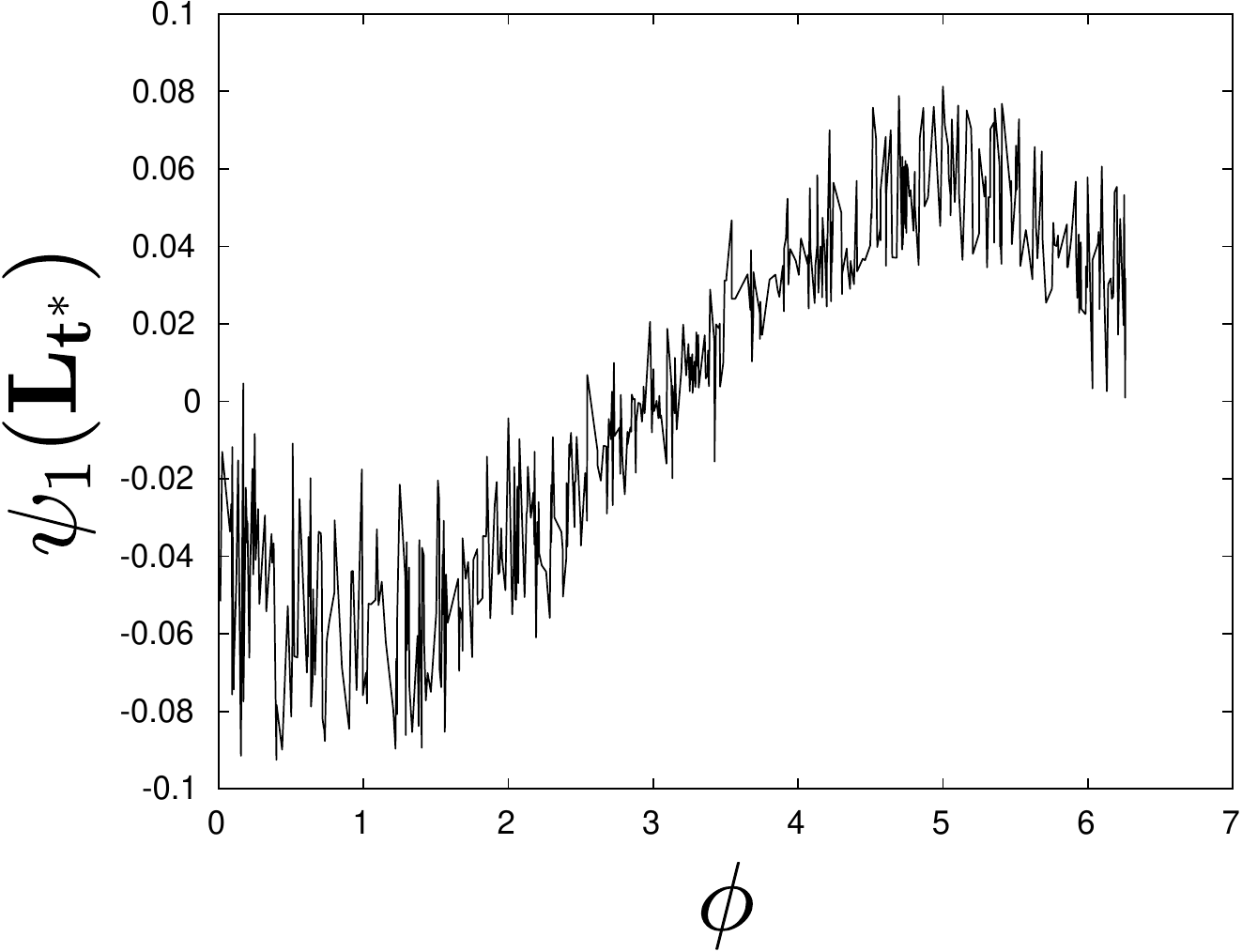} 
\end{tabular}
\caption{Using $s_\mathbf{L}$ (left) and $a_\mathbf{L}$ (right).} 
\label{fig11}
\end{figure}

\begin{remark}[More general definitions of smoothness] \label{gendef}
In this paper, we considered two notions of smoothness based on 
 the quadratic form $\mathbf{x}^\top \mathbf{L} \mathbf{x}$: 
the `smoothest function' normalization, and the 
`average smoothness' normalization, see Remark 
\ref{othernorm}. There are several ways to define intermediate notions of 
smoothness. For example, given weights $w_1,\ldots,w_{n-1}$ one could define 
a notion of smoothness by normalizing by a weighted sum of the eigenvalues:
$$
w_{\mathbf{L}}(\mathbf{x}) = 
 \left( \sum_{j=1}^{n-1} w_j \lambda_j(\mathbf{L}) \right)^{-1}
\mathbf{x}^\top \mathbf{L} \mathbf{x}.
$$
It is also possible to modify the Laplacians used in the definition; since Theorem \ref{thm1} and
Theorem \ref{thm2} only require that
$\mathbf{L}$ is symmetric positive semi-definite, and has eigenvalue $0$ of 
multiplicity $1$ corresponding to constant functions, then it is also possible
to define a notion of smoothness by taking a matrix function of the
graph Laplacian $\mathbf{L}$. For example, for $\alpha > 0$ we could define
$$
s_{\mathbf{L}}^\alpha(\mathbf{x}) = 
\frac{1}{\lambda_1(\mathbf{L})^\alpha} \mathbf{x}^\top 
\mathbf{L}^\alpha \mathbf{x}.
$$
Such a normalization could be used to adjust for different growth rates of
eigenvalues between different graphs. For example, if the given graphs are
approximating manifolds of different dimensions, as in Figure \ref{fig10},
then by Weyl's Law the eigenvalues of the Laplace-Beltrami operator on the 
underlying manifolds will grow at different rates.

\end{remark}

\section{Technical lemma}

\subsection{Bi-stochastic normalization} \label{bistochastic}
We say
that an $n \times n$ nonnegative matrix $\mathbf{B}$ is bi-stochastic if $\mathbf{B} \mathbf{1} = \mathbf{B}^\top \mathbf{1}
= \mathbf{1}$, where $\mathbf{1}$ denotes the $n$-dimensional vector whose entries are all $1$.

\begin{lemma}[Sinkhorn and Kopp \cite{MR210731}] \label{lemsink}
Let $\mathbf{A}$ be an $n \times n$ nonnegative symmetric matrix with a positive main
diagonal. Then, there exists a unique positive definite diagonal matrix $\mathbf{D}$ such
that 
$$
\mathbf{B} = \mathbf{D}^{-1/2} \mathbf{A} \mathbf{D}^{-1/2} 
$$
is a bi-stochastic matrix. Moreover, the matrix $\mathbf{B}$ can be determined by
alternating between normalizing the rows and columns (as detailed below).
\end{lemma}

The iterative procedure of alternating between normalizing the rows and columns
of a symmetric matrix can be expressed as follows. We initialize $\mathbf{Q}_0 = \mathbf{I}$,
 and define
 \begin{equation} \label{eqq}
\mathbf{Q}_{j+1} = \text{diag}\left(\mathbf{A} \mathbf{Q}^{-1}_j \vec{1} \right),
\end{equation}
and set $\mathbf{D} = \lim_{k \rightarrow \infty} \mathbf{Q}_{2k+1} \mathbf{Q}_{2k}$, then $\mathbf{D}^{-1/2} \mathbf{A}
\mathbf{D}^{-1/2}$ will be bi-stochastic. 
Given an adjacency matrix $\mathbf{A}$ satisfying the conditions of Lemma \ref{lemsink}, and unique positive definite matrix $\mathbf{D}$ from Lemma \ref{lemsink}, we define the bi-stochastic graph Laplacian
$\mathbf{L}$ by
$$
\mathbf{L} = \mathbf{I} - \mathbf{D}^{-1/2} \mathbf{A} \mathbf{D}^{-1/2}.
$$

\begin{remark}
The bi-stochastic Laplacian is closely related to other operators such as the graph Laplacian $\mathbf{L} =
\mathbf{Q} - \mathbf{A}$ (where $\mathbf{Q} = \mathbf{A} \mathbf{1}$), the normalized graph Laplacian 
$\mathbf{L} = \mathbf{I} - \mathbf{Q}^{-1/2} \mathbf{A} \mathbf{Q}^{-1/2}$
, and the random walk graph Laplacian $\mathbf{L} = \mathbf{I} - \mathbf{Q}^{-1} \mathbf{A}$. 
The bi-stochastic Laplacian $\mathbf{L}$ is symmetric positive semi-definite, has eigenvector $\mathbf{1}$ of
eigenvalue $0$, and forms a Markov transition matrix when subtracted from the identity matrix.
\end{remark}

\begin{remark}
Numerically, the bi-stochastic graph Laplacian is similar to other graph Laplacians. For
example, if $\mathbf{D} = \mathbf{Q}_0 \mathbf{Q}_1$, where $\mathbf{Q}_j$ is defined above in 
\eqref{eqq}, then $\mathbf{I} - \mathbf{D}^{-1/2} \mathbf{A}
\mathbf{D}^{-1/2}$ is the normalized graph Laplacian. In practice, the
normalized graph Laplacian and random walk graph Laplacian are often used
instead of the standard graph Laplacian. The reason
we use the bi-stochastic Laplacian for our numerical results is that it is closely related 
to the normalized graph Laplacian and random walk graph Laplacian, and has all the properties we require: symmetric positive definite with eigenvalue $0$ of multiplicity $1$ corresponding to constant functions.
\end{remark}

\section{Comments and Remarks}
We conclude with a couple of general comments.

\subsection{Extension to multiple functions}
Recall we defined the smoothest function or common variable $\boldsymbol{\psi}_1$ for a collection of
graph $\mathcal{G}$ by
$$
\boldsymbol{\psi}_1 = 
\argmin_{\mathbf{x} \in X} \max_{k \in \{1,\ldots,m\}} s_{\mathbf{L}_k}(\mathbf{x})  .
$$
By induction we can define
$$
X_k = \{ \mathbf{x} \in X : \boldsymbol{\psi}_1 ^\top \mathbf{x} = \cdots = \boldsymbol{\psi}_k^\top \mathbf{x} = 0 \},
$$
and 
$$
\boldsymbol{\psi}_{k+1} = \argmin_{\mathbf{x} \in X_k} \max_{k \in \{1,\ldots,m\}} s_{\mathbf{L}_k}(\mathbf{x})  ,
$$
for $k = 1,2,3,\ldots$. Informally speaking, $\boldsymbol{\psi}_{k+1}$ is the
smoothest function on $\mathcal{G}$ that is orthogonal to
$\boldsymbol{\psi}_1,\ldots,\boldsymbol{\psi}_k$. Moreover, by restricting the operators $\mathbf{L}_1,\ldots,\mathbf{L}_m$ to $X_k$
the minimax principle of Theorem \ref{thm2} can be applied to solve this optimization problem.
For applications, 
one might suspect that only the first, or possibly two or three, of these functions
would be useful; however, from a theoretical perspective considering
the orthogonal basis $\boldsymbol{\psi}_1,\ldots,\boldsymbol{\psi}_{n-1}$ may be interesting.

\subsection{Sum of diffusions}
We quickly mention
another approach that is quite similar. 
Given graphs $G_1, \dots, G_m$ over the same set of vertices $V$, we can define $m$ different Laplacians $\mathbf{L}_1, \dots, \mathbf{L}_m$ which
we restrict to the space orthogonal to  constants and normalize via
$$ \mathbf{L}_i^* = \frac{1}{\lambda_1(\mathbf{L}_i)} \mathbf{L}_i.$$
This gives rise to diffusion operators 
$$ \mathbf{H}_i(t) = \exp\left(-t  \mathbf{L}_i^* \right).$$
The normalization implies that all these $m$ diffusion operators have the same operator norm
$$ \left\| \mathbf{H}_i(t) \right\| = e^{-t}.$$
If there was a common variable, then it would diffuse slowly among all these different diffusion operators and we would expect that the triangle inequality is almost sharp
$$  \left\|  \sum_{i=1}^{m} \mathbf{H}_i(t) \right\| \leq m e^{-t}.$$
This allows us to define a numerical score 
$$ 1 \leq e^t \left\|  \sum_{i=1}^{m} \mathbf{H}_i(t) \right\| \leq m$$
measuring how many of these graphs do indeed have a common variable. Naturally, $\mathbf{H}_i(t)$ will be close to the identity for $t$ small, so the inequality becomes more interesting for $t$ large (and $t$ can play the role of a consistency parameter). This may be interpreted as a simple `one-shot' version of our main idea.

\subsection{Random Matrices}
We note that the dual version of this idea, the matrix exponential of a linear
combination of Laplacians as opposed to a linear combination of matrix
exponentials of Laplacians, has a probabilistic interpretation.
When we consider applying small multiples of random Laplacians, the main
question is the following: if $\mathbf{X}_1, \dots, \mathbf{X}_m \in \mathbb{R}^{n \times n}$ are
$m$ matrices what can be said about products
$$
\mathbf{X}_s = \left(\mathbf{I} + \frac{\varepsilon}{s} \mathbf{X}_{i_1}\right) 
\left(\mathbf{I} + \frac{\varepsilon}{s} \mathbf{X}_{i_2}\right) \dots 
\left(\mathbf{I} + \frac{\varepsilon}{s} \mathbf{X}_{i_s}\right),
$$
as $s \rightarrow \infty$? Here we think of $i_j$ as randomly (independently
and uniformly) chosen elements from $\left\{1, 2, \dots, m\right\}$. An even
more general question was studied by Emme and Hubert \cite{emme} whose result implies that
$$ 
\lim_{s \rightarrow \infty} \mathbf{X}_s = \exp\left( \frac{\varepsilon}{m}
\sum_{k=1}^{m} \mathbf{X}_k \right).
$$
In our setting, we note that
$$   
\exp\left(-\frac{\varepsilon t_k^*}{\lambda_1(\mathbf{L}_k)} \mathbf{L}_k\right) =
\mathbf{I} - \frac{\varepsilon t_k^*}{\lambda_1(\mathbf{L}_k)}
\mathbf{L}_k + \mathcal{\boldsymbol{O}}(\varepsilon^2)
$$
implying that the proper limit of random products of matrix exponentials of
suitably rescaled Laplacians would result in
$$ 
\lim_{s \rightarrow \infty} \mathbf{X}_s = \exp \left( -\frac{\varepsilon}{m}
\mathbf{L}_{\mathbf{t}^*} \right),
$$
which is a natural variant of our approach.

\subsection{Summary and discussion} We repeat the main problem:
suppose we are given a collection $\mathcal{G}$ 
of
$m$ different graphs over the same set of vertices $V$
$$
\mathcal{G} = \{ G_1=(V, E_1),  \ldots, G_m = (V, E_m) \}.
$$
Among all nonconstant functions $f : V \rightarrow \mathbb{R}$ which is the `smoothest' 
with respect to $\mathcal{G}$?
We believe this problem to be of substantial interest. Naturally, there is a
certain vagueness in how the problem is posed: 1) what does it mean for a
function to be smooth? and 2) what does it mean for a function to be commonly
smooth? 

In this paper, we propose the classical spectral definition for 1) and
a minimax approach for 2). One could, naturally, consider a great many other
approaches and we believe it to be a fascinating question for further study.
For example, the maximum norm in the definition of smoothness could be replaced by an $\ell^p$ norm
$$
\boldsymbol{\psi}_1 = \argmin_{\mathbf{x} \in X} \left\|
(s_{\mathbf{L}_1}(\mathbf{x}), \ldots, s_{\mathbf{L}_k}(\mathbf{x}))
\right\|_{\ell^p},
$$
for some $1 \le p \le \infty$, alternatively, the quadratic form in smoothness score
$s_{\mathbf{L_k}}(\mathbf{x})$ could be replaced by a different quantity, for example, 
the Laplacians could be taken to a power as discussed in Remark \ref{gendef}.
In summary, there are many potentially interesting ways to formalize our main
question; the minimax theorem established in this paper solves the problem for 
a specific notion of smoothness inspired by spectral graph theory, and provides
a basis for further work.

\end{document}